                    \def\version{July 20, 2016}                       %

 \documentclass[reqno,11pt]{amsart}
 \usepackage{amsmath, amsthm, a4, latexsym, amssymb}
\usepackage[unicode]{hyperref}
\usepackage{srcltx}
\usepackage{xcolor}

\def\@rmrk#1#2{\refstepcounter
    {#1}\@ifnextchar[{\@yrmrk{#1}{#2}}{\@xrmrk{#1}{#2}}}

\makeatletter\@addtoreset{equation}{section}\makeatother

 \newfont{\bfit}{cmbxti10 scaled 1200}

\renewcommand{\d}{{\rm d}}
 \newcommand{\e}{{\rm e} }

 \newcommand{\eps}{\varepsilon}

 \newcommand{\dist}{{\rm dist}}

 \newcommand{\diam}{{\rm diam}}

 \newcommand{\R}{\mathbb{R}}
 \newcommand{\N}{\mathbb{N}}
 \newcommand{\Z}{\mathbb{Z}}

 \newcommand{\E}{\mathbb{E}}
 \renewcommand{\P}{\mathbb{P}}
 \renewcommand{\thefootnote}{\arabic{footnote}}
 \def\1{{\mathchoice {1\mskip-4mu\mathrm l} 
{1\mskip-4mu\mathrm l}
{1\mskip-4.5mu\mathrm l} {1\mskip-5mu\mathrm l}}}

 \newcommand{\Mcal}{{\mathcal M}}

\newcommand{\heap}[2]{\genfrac{}{}{0pt}{}{#1}{#2}}

\newcommand{\ssup}[1] {{\scriptscriptstyle{({#1}})}}

\renewcommand{\subsection}{\secdef \subsct\sbsect}
\newcommand{\subsct}[2][default]{\refstepcounter{subsection}
\vspace{0.15cm}
{\flushleft\bf \arabic{section}.\arabic{subsection}~\bf #1  }
\nopagebreak\nopagebreak}
\newcommand{\sbsect}[1]{\vspace{0.1cm}\noindent
{\bf #1}\vspace{0.1cm}}

{\nopagebreak {\hfill\rule{2mm}{2mm}}\\ }

\newtheorem{theorem}{Theorem}[section]
\newtheorem{lemma}[theorem]{Lemma}
\newtheorem{cor}[theorem]{Corollary}
\newtheorem{prop}[theorem]{Proposition}

\newtheoremstyle{thm}{1.5ex}{1.5ex}{\itshape\rmfamily}{}
{\bfseries\rmfamily}{}{2ex}{}

\newtheoremstyle{rem}{1.3ex}{1.3ex}{\rmfamily}{}
{\itshape\rmfamily}{}{1.5ex}{}
\theoremstyle{rem}

\refstepcounter{subsubsection}

\def\thebibliography#1{\section*{References}
  \list%
  {\arabic{enumi}.}
    {\settowidth\labelwidth{[#1]}\leftmargin\labelwidth
    \advance\leftmargin\labelsep
    \parsep0pt\itemsep0pt
    \usecounter{enumi}}
    \def\newblock{\hskip .11em plus .33em minus .07em}
    \sloppy                   
    \sfcode`\.=1000\relax}

\begin{document}
\title[A rigorous construction of the Pekar process]
{\large Mean-field interaction of Brownian occupation measures. II: A rigorous construction of the Pekar process}
\author[Erwin Bolthausen, Wolfgang K\"onig and Chiranjib Mukherjee ]{}
\maketitle
\thispagestyle{empty}
\vspace{-0.5cm}

\centerline{\sc Erwin Bolthausen\footnote{Institut f\"ur Mathematik, Universit\"at Z\"urich, Winterthurerstrasse 190, Zurich 8057, {\tt eb@math.uzh.ch}}, 
Wolfgang K\"onig\footnote{TU Berlin and WIAS Berlin, Mohrenstra{\ss}e 39, 10117 Berlin, {\tt koenig@wias-berlin.de}} 
and Chiranjib Mukherjee\footnote{Courant Institute of Mathematical Sciences, 251 Mercer Street, New York 10112, {\tt mukherjee@cims.nyu.edu},  on leave from WIAS}}
\renewcommand{\thefootnote}{}

\vspace{-0.5cm}
\centerline{\textit{University of Zurich, WIAS Berlin and TU Berlin, Courant Institute New York and WIAS Berlin}}
\vspace{0.2cm}

\begin{center}
\version
\end{center}

\begin{abstract}

We consider mean-field interactions corresponding to Gibbs measures on interacting Brownian paths in three dimensions. The interaction is self-attractive and is given by a singular Coulomb 
potential. The logarithmic asymptotics of the partition function for this model were identified in the 1980s by Donsker and Varadhan \cite{DV83-P} in terms of the {\it{Pekar variational formula}},
which coincides with the behavior of the partition function of the {\it{polaron problem}} under strong coupling. Based on this, in 1986 Spohn (\cite{Sp87}) made a heuristic observation 
that the strong coupling behavior of the polaron path measure, on certain time scales, should  resemble a process, named as the {\it{Pekar process}}, whose distribution could somehow be guessed from the limiting asymptotic behavior of the mean-field measures under interest, whose rigorous analysis remained open. The present paper is devoted to a precise analysis of these mean-field path measures and
convergence of the normalized occupation measures towards an explicit mixture of the maximizers of the Pekar variational problem.
This leads to a rigorous construction of the aforementioned Pekar process and hence, is a contribution to the understanding of the 
``mean-field approximation" of the polaron problem on the level of path measures.

The method of our proof is based on the compact large deviation theory developed in \cite{MV14}, its extension to the uniform strong metric 
for the singular Coulomb interaction carried out  in \cite{KM15}, as well as an idea inspired by a {\it{partial path exchange}} argument
appearing in \cite{BS97}. 

\end{abstract}

\maketitle   

\maketitle   



\section{Motivation and introduction}\label{intro}

\subsection{Motivation.} Questions on path measures pertaining to self-attractive random motions, or Gibbs measures on interacting random paths, are often motivated by the important r\^{o}le 
they play in quantum statistical mechanics. A problem similar in spirit to these considerations is connected with the {\it polaron problem}. 
The physical question arises from the description of the slow movement of a charged particle, e.g. an electron, in a crystal whose lattice sites are polarized by this motion, 
influencing the behavior of the electron and determining its {\it{effective behavior}}. For the physical relevance of this model, we refer to
the lectures by Feynman \cite{F72}. The mathematical layout of this problem was also founded by Feynman. Indeed, he introduced a path integral formulation of this problem and pointed out that the aforementioned effective behavior can be studied via studying a certain path measure. This measure is written in terms of a three dimensional
Brownian motion acting under a self-attractive Coulomb interaction: 
$$
\widehat\P_{\lambda,t}(\d \omega)= \frac 1 {Z_{\lambda,t}} \exp\bigg\{\lambda \int_0^t\int_0^t \d\sigma \d s \frac{\e^{-\lambda|\sigma-s|}}{|\omega_\sigma- \omega_s|}\bigg\} \,\, \P(\d \omega),
$$
where $\lambda>0$ is a parameter, $\P$ refers to the three dimensional Wiener measure and $Z_{\lambda,t}$ is the normalization constant or partition function. One calls $\alpha=1/\sqrt{\lambda}$ the {\it{coupling parameter}}. The physically relevant regime is the  {\it{strong coupling limit}} as $\alpha\to \infty$, i.e., $\lambda\to0$.

We remark that the above interaction is {\it{self-attractive}}, as the asymptotic behavior of the path measure is essentially determined by those
paths which make $|\omega_\sigma- \omega_s|$ small, when $|\sigma- s|$ is also small. In other words, these paths tend to clump together on short time scales. 

The asymptotic behavior of the partition function $Z_{\lambda,t}$ in the limit $t\to\infty$, followed by $\lambda\to0$, was rigorously studied  by Donsker and Varadhan (\cite{DV83-P}). The intuition says that, for small $\lambda$, the
interaction should get more and more smeared out and should approach the {\it mean-field interaction} 
$$
\frac 1t\int_0^t\int_0^t \d \sigma\d s \,\frac 1{\big|\omega_\sigma-\omega_s\big|}.
$$
Donsker and Varadhan proved this intuition on the level of the logarithmic large $t$-asymptotics for the partition function leading to an explicit variational formula. More precisely, they showed that that the large-$t$ asymptotics of $Z_{\lambda,t}$ coincides in the limit $\lambda\to 0$ with the large-$t$ asymptotics of
the {\it mean field partition function} 
\begin{equation}\label{partfn}
Z_t= \E\bigg[ \exp\bigg\{\frac 1t\int_0^t\int_0^t \d \sigma\d s \,\frac 1{\big|W_\sigma-W_s\big|}\bigg\}\bigg],
\end{equation}
where $(W_s)_{s\in[0,\infty)}$ denotes a three-dimensional standard Brownian motion. In other words, they proved {\it Pekar's conjecture} (\cite{P49}) and derived the following identification of the free energy variational formula:
\begin{equation}\label{Pekarfor}
 \begin{aligned}
\lim_{\lambda\to 0}\lim_{t\to\infty}\frac 1t\log Z_{\lambda,t}&= \lim_{t\to\infty}\frac 1t\log Z_t \\
&=\sup_{\heap{\psi\in H^1(\R^3)}{\|\psi\|_2=1}} 
\Bigg\{\int_{\R^3}\int_{\R^3}\d x\d y\,\frac {\psi^2(x) \psi^2(y)}{|x-y|} -\frac 12\big\|\nabla \psi\big\|_2^2\Bigg\},
\end{aligned}
\end{equation}
with $H^1(\R^3)$ denoting the usual Sobolev space of square integrable
functions with square integrable gradient. The variational formula in \eqref{Pekarfor} was analyzed by Lieb (\cite{L76}) with the result that there is a rotationally symmetric maximizer $\psi_0$, which is unique modulo spatial shifts. 

Given  \eqref{Pekarfor}, it is natural to guess that 
the polaron path measure should somehow be related to the mean-field path measure, given by
\begin{equation}\label{Phat}
\widehat{\P}_t(\d \omega)=\frac 1 { Z_t}\, \exp\bigg\{\frac 1t\int_0^t\int_0^t \d \sigma\d s \,\frac 1{\big|\omega_\sigma-\omega_s\big|}\bigg\} \, \P(\d \omega). 
\end{equation}
However, even a clear formulation of such a relation is by far not obvious. Spohn (\cite{Sp87}) presented a heuristic analysis of the effective behavior of the polaron measure $\widehat\P_{\lambda,t}$ for $\lambda\sim 0$,
whose rigorous asymptotic analysis remains open. The heuristic discussion in \cite{Sp87} is based on the idea that, for the strong coupling regime (i.e, $\lambda\to 0$), on time scales of order $\lambda^{-1}$, the polaron measure $\widehat\P_{\lambda,t}$ should resemble a process, named as the {\it{Pekar process}}, whose empirical distribution can be guessed from the limiting asymptotic behavior of $\widehat\P_t\circ L_t^{-1}$, the distributions of the normalized Brownian occupation measures $L_t=\frac 1t\int_0^t\d s\, \delta_{W_s}$, under the mean-field transformations $\widehat\P_t$, whose rigorous analysis was also left open. 

This paper is devoted to a precise analysis of the limiting behavior 
of the distributions of mean-field path measures $\widehat\P_t\circ L_t^{-1}$ and a rigorous construction of the law of the aforementioned Pekar process. Hence, it is a contribution to the understanding of the mean-field approximation of the polaron problem on the level of path measures.

\subsection{The model and the problem.} 

\noindent We turn to a precise formulation and the mathematical layout of the problem. We start with the Wiener measure $\P$ on $\Omega=C([0,\infty),\R^3)$ corresponding to a $3$-dimensional Brownian motion $W=(W_t)_{t\geq 0}$ starting from the origin. 
 Let  
\begin{equation}\label{OcMeas}
L_t = \frac 1t \int_0^t  \d s \, \delta_{W_s}
\end{equation}
be the normalized occupation measure of $W$ until time $t$. This is a random element of $\Mcal_1(\R^3)$, the space of probability measures on $\R^3$. Then the mean-field path measure $\widehat{\P}_t$ defined in \eqref{Phat} can be written as 
$$
\widehat{\P}_t(A)=\frac 1{Z_t}\, \E\big[\1_A\, \exp\big\{t H(L_t)\big\}\big], \qquad A\subset \Omega,
$$
where $Z_t$ is the normalizing constant defined in \eqref{partfn} and 
\begin{equation}\label{Hdef}
H(\mu)= \int_{\R^3}\int_{\R^3}  \frac {\mu(\d x)\,\mu(\d y)}{|x-y|},\qquad \mu\in\Mcal_1(\R^3),
\end{equation}
denotes the {\it{Coulomb potential energy functional}} of $\mu$, or the Hamiltonian. It is the goal of the present paper to analyze and identify the limiting distribution 
of $L_t$ under $\widehat\P_t$.

 We recall that the partition function $Z_t=\E\big[\exp\big\{t H(L_t)\big\}\big]$, which is finite in $\R^3$, was analyzed by Donsker and Varadhan \cite{DV83-P} resulting in the variational formula 
\eqref{Pekarfor}. For future reference, let us write
 \begin{equation}\label{rhodef}
 \begin{aligned}
 \rho= \sup_{\mu\in \Mcal_1(\R^3)} \bigg\{ H(\mu)- I(\mu)\bigg\}
=\sup_{\heap{\psi\in H^1(\R^3)}{\|\psi\|_2=1}} 
\Bigg\{\int_{\R^3}\int_{\R^3}\d x\d y\,\frac {\psi^2(x) \psi^2(y)}{|x-y|} -\frac 12\big\|\nabla \psi\big\|_2^2\Bigg\},
\end{aligned}
\end{equation}
where we introduced the functional 
\begin{equation}\label{Idef}
I(\mu)=\frac 12\|\nabla\psi\|_2^2
\end{equation}
if $\mu$ has a density $\psi^2$ with $\psi\in H^1(\R^3)$, and $I(\mu)=\infty$ otherwise. 
Note that this is the rate function for the classical weak large deviation theory for the distribution of $L_t$ under $\P$, developed by Donsker and Varadhan (\cite{DV75-1}- \cite{DV83-4}). We remark that both $H$ and $I$ are shift-invariant functionals, i.e., $H(\mu)=H(\mu\star\delta_x)$  and $I(\mu)=I(\mu\star\delta_x)$ for any $x\in\R^3$. 

We also recall (\cite{L76}) that the variational formula \eqref{rhodef} possesses a smooth, rotationally symmetric and centered maximizer $\psi_0$, which is unique modulo spatial translations. In other words, if $\mathfrak m$ denotes the set of maximizing densities, then
\begin{equation}\label{shiftunique}
\mathfrak m=  \big\{\mu_0 \star \delta_x\colon x\in \R^3\big\},
\end{equation}
where $\mu_0$ is the probability measure with density $\psi_0^2$. We will often write $\mu_x=\mu_0\star \delta_x$ and write $\psi_x^2$ for its density.
 
Note that given \eqref{Pekarfor} and \eqref{shiftunique}, we expect the distribution of $L_t$ under the transformed measure $\widehat \P_t$ to concentrate around $\mathfrak m$ and, even more, to converge towards 
a mixture of spatial shifts of $\mu_0$, thanks to the uniqueness statement \eqref{shiftunique} for the free energy variational problem \eqref{rhodef}. Such a precise analysis was carried out by Bolthausen and Schmock \cite{BS97} for a spatially discrete version of $\widehat\P_t$, i.e., for the continuous-time simple random walk on $\Z^d$ instead of Brownian motion and a bounded interaction potential $V\colon \Z^d \rightarrow [0,\infty)$ with finite support instead of the singular Coulomb potential $x\mapsto 1/|x|$. A first key step in \cite{BS97} was to show that, under the transformed measure, the probability of the local times falling outside any neighborhood of the maximizers decays exponentially. For its proof, the lack of a strong large deviation principle (\cite{DV75-1}-\cite{DV83-4}) for the local times was handled by an extended version of a standard periodization procedure by folding the random walk into some large torus. Combined with this, an explicit tightness property of the distributions of the local times led 
to an identification of the limiting distribution.

However, in the continuous setting with a singular Coulomb interaction, the aforementioned periodization technique or any standard compactification procedure does not work well to circumvent the lack of a strong large deviation principle. An investigation of $\widehat\P_t \circ L_t^{-1}$, the distribution of  $L_t$ under $\widehat \P_t$, remained open until a recent result \cite{MV14} rigorously justified the above heuristics, leading to the statement 
\begin{equation}\label{tubeweak}
\limsup_{t\to\infty} \frac 1t \log \widehat\P_t \big\{L_t \notin U(\mathfrak m)\big\}<0,
\end{equation}
where $U(\mathfrak m)$ is any neighborhood of $\mathfrak m$ in the weak topology induced by the Prohorov metric, the metric that is induced by all the integrals against continuous bounded test functions. \eqref{tubeweak} implies that the distribution of $L_t$ under $\widehat \P_t$ is asymptotically concentrated around $\mathfrak m$. Since a one-dimensional picture of $\mathfrak m$ reflects an infinite line, its neighborhood resembles an infinite tube. Therefore, assertions similar to \eqref{tubeweak} are referred to as {\em tube properties}. 

It is worth pointing out that although \eqref{tubeweak} requires only the weak topology in the statement, its proof is crucially based on a robust theory of {\it{compactification}} $\widetilde {\mathcal X}$ of the quotient space
$$
\widetilde\Mcal_1(\R^d) \hookrightarrow \widetilde {\mathcal X}
$$ 
of orbits $\widetilde \mu=\{\mu \star\delta_x\colon x\in\R^3\}$ of probability measures $\mu$ on $\R^d$ under translations and a full large deviation principle for the distributions of $\widetilde L_t\in \widetilde\Mcal_1(\R^d)$ embedded in the compactification. In particular, this is based on a topology induced by a different metric in the compactfication $\widetilde {\mathcal X}$. However, the statement \eqref{tubeweak} simply drops out from this abstract
set up, thanks to the shift invariance of the Hamiltonian $H(\mu)$ as well as the rate function $I(\mu)$.

Note that, even after proving \eqref{tubeweak}, since $\mathfrak m$ and hence any neighborhood of this shift-invariant set is non-compact, the occupation measures $L_t^{-1}$ under $\widehat\P_t$ could still 
fluctuate wildly in the infinite tube. The crucial result of the present article says that this can happen only with small $\widehat\P_t$-probability and that the distributions $\widehat\P_t \circ L_t^{-1}$ converge weakly to an explicit mixture of the elements in $\mathfrak m$. The process corresponding to this mixture is the aforementioned Pekar process.

\section{Main results}

\noindent We turn to the statements of our main results. 

\subsection{Convergence of $\boldsymbol{\widehat\P_t \circ L_t^{-1}}$ and identification of the limit.}

\noindent Recall that $\psi_0$ is the  unique radially symmetric and centered maximizer of the Pekar variational problem \eqref{rhodef} and that $\mu_x$ denotes
the probability measure with density $\psi_x^2= \psi_0^2 \star \delta_x$. Here is the main result of the present paper.

\begin{theorem}\label{thm1}
Let $\widehat{\mathbb Q}_t$ denote the distribution of $L_t$ under $\widehat\P_t$. Then,
\begin{equation}\label{eqthm1}
\lim_{t\to\infty}\widehat{\mathbb Q}_t = \frac{\int_{\R^3} \d x \, \, \psi_0(x) \delta_{\mu_x}}{\int_{\R^3} \d x \,\,\psi_0(x)},
\end{equation}
weakly as probability measures on $\Mcal_1(\R^3)$.
\end{theorem}

In words, the distribution of the occupation measures under $\widehat\P_t$ converges to a random spatial shift of the maximizer $\psi_0^2$, and the distribution of this shift is $\psi_0(x)\,\d x$, properly normalized. The proof of Theorem~\ref{thm1} can be found in Section \ref{sectionthm1proof}. For a heuristic sketch, we refer to Section \ref{heuristicthm2}. 
We also remark that, following exactly the same line of arguments appearing in the proof of Theorem \ref{thm1}, one can derive the asymptotic behavior
of the distribution $\widehat\P_t \circ W_t^{-1}$ of the end point of the path $W_t$ under $\widehat\P_t$. The limiting distribution is equal to $(\psi_0\star\psi_0)(x)\,\d x\big/\int_{\R^3}\d y \, (\psi_0\star\psi_0)(y)$.

{\bf{Discussion on the construction of the Pekar process.}} Let us now briefly comment on the limiting behavior of $\widehat\P_t$ itself,
which drops out from the main steps for the proof of Theorem \ref{thm1}.
This limiting assertion is directly related to the interpretation of the aforementioned Pekar process, 
Let us first remark on the heuristic definition of the Pekar process set forth by Spohn in \cite{Sp87}.
For any $\mu\in\Mcal_1(\R^3)$, let
$$
\big(\Lambda\mu\big)(x)= \bigg(\mu\star \frac 1{|\cdot|}\bigg)(x)= \int_{\R^3} \frac{\mu(\d y)}{|x-y|}
$$
be the smooth {\em Coulomb functional} of $\mu$. Then via the Feynman-Kac formula corresponding to the semigroup of $\frac 12 \Delta+ \Lambda\psi_0^2$, one can construct the measures
\begin{equation}\label{Pekarmeasure}
\frac 1 {Z_t^{\ssup{\psi_0}}} \exp\bigg\{\int_0^t \d s \,  (\Lambda\psi_0^2)(W_s) \bigg\}\,\d\P.
\end{equation}
These probability measures then should converge, as $t\to\infty$, towards a measure which governs the law of a stationary diffusion process $(X_s^{\mathrm{\ssup{Pek}}})_{s\in[0,\infty)}$,
 driven by the stochastic differential equation (SDE)
$$
\d X_t^ {\mathrm{\ssup{Pek}}}= \d W_t+ \bigg(\frac{\nabla\psi_0}{\psi_0}\bigg)(W_t) \, \d t.
$$
Since $\psi_0$ is centered, the drift in the above SDE points towards the origin and suppresses large fluctuations of the diffusion. Furthermore, the diffusion process is ergodic with invariant measure $\psi_0^2$. 

In the light of the above  discussion, our main result has an interesting consequence on a rigorous level: the proof of Theorem \ref{thm1}
reveals what the precise limit of the measures $\widehat\P_t$ itself should be.
Indeed, this is a spatially inhomogeneous mixture of Markovian path measures with generators 
$$
\frac 12 \Delta+ \frac{\nabla\psi_x}{\psi_x}\cdot \nabla,\qquad x\in\R^3,
$$
with the spatial mixture being taken w.r.t.~ the measure $\psi_0(x)\,\d x/\int \d y \, \psi_0(y)$. This assertion actually carries the flavor
of the classical Gibbs conditioning principle. Given the salient features constituting the proof of Theorem \ref{thm1} (see Section 
\ref{heuristicthm2} for a heuristic sketch), a complete proof of this assertion therefore follows a routine method. To avoid repetition, we refrain from 
spelling out the details and content ourselves only with its statement, which 
clearly underlines the rigorous interpretation of the law of the Pekar process as the aforementioned spatial mixture of Markovian path measures and justifies the heuristic discussion in \cite{Sp87}.


\subsection{Earlier results: Tube properties under $\boldsymbol{\widehat\P_t}$.}\label{sec-Earlier}

\noindent We collect here some crucial results from \cite{MV14} and \cite{KM15} (see also \cite{M15}) that will be used in the sequel.

As mentioned before, a fundamental first step is to 
prove that, under $\widehat\P_t$, the occupation measures $L_t$ concentrate with high probability in any neighborhood of $\mathfrak m$.
This has been rigorously justified in Theorem 5.1 in \cite{MV14}:

\begin{theorem}[Tube property under the weak topology]\label{thm-tubeweak}
For any neighborhood $U(\mathfrak m)$ of $\mathfrak m$ in the weak topology in $\Mcal_1(\R^3)$,
$$
\limsup_{t\to\infty}\frac 1t \log \widehat\P_t\{L_t \notin U(\mathfrak m)\}<0.
$$
\end{theorem}

In \cite{KM15}, this tube property in the weak topology has been strengthened in a ``functional form" in the strong topology under the uniform metric. Note that $H(\mu)= \big\langle \mu, \Lambda \mu\big\rangle= \int (\Lambda\mu)(x) \,\mu(\d x)$, where we recall the Coulomb functional $\Lambda(\mu)$ of $\mu$. We remark that the Coulomb energy of the Brownian occupation measure,  
\begin{equation}\label{Lambdadef}
\Lambda_t(x)=\big(\Lambda L_t\big)(x)= \int_{\R^3} \frac{L_t(\d y)}{|x-y|}= \frac 1 t \int_0^t \frac{\d s}{|W_s- x|},
\end{equation}
is almost surely finite in $\R^3$. 

Let us write $\Lambda(\psi^2)(x)=\int \d y\,\frac{\psi^2(y)}{|x-y|}$ for functions $\psi^2$, and recall that $\psi_w^2= \psi_0^2 \star \delta_w$ denotes the shift of the maximizer $\psi_0^2$ of the second variational formula \eqref{rhodef} by $w\in\R^3$. The following theorem was coined as the tube property for $\Lambda_t$ in the uniform metric (see Theorem 1.1, \cite{KM15}).

\begin{theorem}[Tube property under the uniform metric]\label{tubeunifmetric}
For any $\eps>0$, 
\begin{equation}\label{eqtubeunifmetric}
\limsup_{t\to\infty}\frac 1t\log \widehat\P_t\bigg\{\inf_{w\in \R^3} \big\|\Lambda_t- \Lambda \psi_w^2\big\|_\infty > \eps\bigg\} <0.
\end{equation}
\end{theorem}

As a consequence of Theorem \ref{tubeunifmetric}, the Hamiltonian $H(L_t)=\langle L_t,\Lambda L_t\rangle$ converges in distribution towards the common Coulomb energy of any member of $\mathfrak m$:

\begin{cor}\label{convergeHLt}
Under $\widehat\P_t$, the distributions of $H(L_t)$ converge weakly to the Dirac measure at
$$
H(\psi_0^2)= \int\int_{\R^3\times\R^3} \frac{\psi_0^2(x)\psi_0^2(y)}{|x-y|}\, \d x\d y.
$$
\end{cor}

Finally, we state an interesting fact concerning the exponential decay of the probability of deviations of the uniform norm of $\Lambda_t$ from zero.
This fact followed from the exponential regularity estimates for $\Lambda_t$ in the uniform norm derived in \cite{KM15}, see Theorem 1.3 and Corollary 1.4 in \cite{KM15}.
The following proposition will be used often in the sequel. 

\begin{prop}\label{expdecaysupnorm}
 For any $a>0$,
 $$
 \limsup_{t\to\infty} \frac 1t \log \P\big\{\|\Lambda_t\|_\infty>a\big\}<0.
 $$
 \end{prop}

\subsection{Tightness of the distributions of $\boldsymbol{L_t}$ under $\boldsymbol{\widehat\P_t}$.}

\noindent We now turn to the second main ingredient that constitutes the proof of Theorem \ref{thm1}.

We emphasize that the localization properties stated in Section~\ref{sec-Earlier} still allows the measures $L_t$  to {\it a priori} float around freely in the infinite tubular neighborhood of $\mathfrak m$.
The next main step is to justify that this can happen only with small $\widehat\P_t$ probability, and this is our second main result.


\begin{theorem}[Tightness]\label{thm2}
For any $\eta>0$, there exists $k(\eta)>0$ such that, for any neighborhood $U$ of the set $\{\mu_x\colon x\in\R^3, |x|\leq k(\eta)\}$, 
$$
\limsup_{t\to\infty}\widehat\P_t\{L_t\notin U\}< \eta.
$$
\end{theorem}

We will actually prove a slightly weaker version of Theorem \ref{thm2}. Combined with Theorem \ref{tubeweak}, this result will clearly imply Theorem \ref{thm2}.
 \begin{theorem}\label{thm3}
 There exists $\eps_0>0$ such that for all $\eps\leq \eps_0$ and for all $\eta>0$, there exist $k{(\eps,\eta)}>0$ and $u(\eps,\eta)>0$ so that
 $$
 \widehat\P_t\bigg\{ \bigcup_{|x|\geq k{(\eps,\eta)}} \big\{L_t\in U_\eps(\mu_x)\big\}\bigg\} <\eta,
 $$
 for all $t\geq u(\eps,\eta)$.
\end{theorem}


Section \ref{sectionthm2proof} is devoted to the proof of Theorem \ref{thm3}. 
The final step to prove Theorem \ref{thm1} is then to justify the weak convergence of the measures $\widehat{\mathbb Q}_t$ to the 
limiting measure appearing on \eqref{eqthm1}. This follows from Theorem \ref{tubeweak} and Theorem \ref{thm2}. Section \ref{sectionthm1proof} constitutes the proof of Theorem \ref{thm1}. 
Let us now heuristically sketch the central idea behind the proof of Theorem \ref{thm3}, which contains the heart of the argument for Theorem \ref{thm1}.

\subsection{Heuristic proof of Theorem \ref{thm3} and Theorem \ref{thm1}: Partial path exchange.}\label{heuristicthm2}

\noindent We now heuristically justify that, under $\widehat \P_t$, $L_t$ can not build up its mass over a long time close to some $\mu_x$ if $x$ is far away from the starting point
of the path. To estimate this event, we study the ratio
\begin{equation}\label{ratio}
\frac {\widehat \P_t(L_t\approx \mu_x)}{\widehat \P_t(L_t\approx \mu_0)}= \frac {\E\big\{\exp\{tH(L_t)\}\,\1_{\{L_t\approx \mu_x\}}\big\}}{\E\big\{\exp\big\{tH(L_t)\big\}\,\1_{\{L_t\approx \mu_0\}}\big\}},
\end{equation}
and show that this is small as $t\to\infty$. For this, we emulate an approach similar to \cite{BS97}, which resembles the Peierls argument in the Ising model.

It is conceivable that for the event $\{L_t\approx \mu_x\}$ to happen, the path, starting from the origin, reaches a neighborhood of $x$ relatively quickly, say by time $t_0$, and concentrates in that neighborhood for the remaining time time $t-t_0$. This leads to a splitting of the occupation measure 
$$
L_t= \frac {t_0} t L_{t_0}+ \frac {t-t_0}t L_{t_0,t}
$$ 
before and after time $t_0$. 
Although we will choose $t_0\ll t$, note that the time $t_0$ should also get large as $|x|$ gets large. 
Note that the above splitting of $L_t$ also leads to the splitting of the Hamiltonian:
\begin{equation}\label{eqheur0}
t H(L_t)= \frac{t_0^2}{t}\, H(L_{t_0})+2 \frac{t_0(t-t_0)}{t}\big\langle L_{t_0} , \Lambda_{t_0,t}\big\rangle+ \frac{(t-t_0)^2}{t}H(L_{t_0,t}),
\end{equation}
where $\Lambda_{t_0,t}(x)= \int_{\R^3} \frac{L_{t_0,t}(\d y)}{|x-y|}$ is the Coulomb functional of $L_{t_0,t}$. Since $t_0\ll t$, the first term on the right-hand side should be negligible.


Let us turn to the second term, which makes a difference in the ratio \eqref{ratio} with the occurrences of 
$L_t\approx\mu_x$, respectively $L_t\approx \mu_0$, and makes the ratio \eqref{ratio} small as follows. 
Note that with high probability, on $\{L_t\approx \mu_x\}$, we expect that $L_{t_0,t}\approx \mu_x$ in the weak topology. Hence, by Theorem~\ref{tubeunifmetric}, 
$\Lambda_{t_0,t}\approx \Lambda(\mu_x)$ with high $\widehat\P_t$-probability in the uniform strong topology. 
Consequently, the second term (in the numerator in \eqref{ratio}) can essentially be replaced by $2t_0 \big\langle L_{t_0} ,\Lambda(\mu_x)\big\rangle$. But $\Lambda(\mu_x)$, being concentrated around $x$, has only vanishing interaction with $L_{t_0}$ as $|x|$ is large. On the contrary, in the denominator in \eqref{ratio} (i.e., on the event $\{L_{t_0}\approx \mu_0\}$), by the same reasoning
 we have $\Lambda_{t_0,t}\approx \Lambda(\mu_0)$, which has a non-trivial interaction with $L_{t_0}$. Hence, the difference made by the second term appearing in the numerator and the denominator in \eqref{ratio} can be quantified as $-C t_0$ with some $C>0$ that depends only on $\psi_0$, and $t_0$ is large if $|x|$ is large.

The third term on the right hand side of \eqref{eqheur0} involves only the path on the time interval $[t_0,t]$ when it hangs around $x$ (for the numerator) respectively around $0$ (for the denominator). 
To compare these two scenarios, we just shift the path in the numerator (when it is close to $x$) after time $t_0$ by the amount of $-x$. 
The crucial upshot is, due to the shift-invariance of the Hamiltonian $H(L_{t_0,t})$, which is being exploited heavily, this {\it{path exchange argument}} does not cost anything to the main bulk of the path
in the long time interval $[t_0,t]$. If we call the occupation measure of the shifted path $L_t^{\ssup {\mathrm{shift}}}$, the only
feature that essentially distinguishes $H(L_t)$ from $H(L_t^{\ssup {\mathrm{shift}}})$ is that the former has essentially no contribution from the interaction between the path on $[0, t_0]$ with the path on $[t_0, t]$, where the latter has. For our purposes, we need to quantify this interaction, for which we switch from our original measure $\P$ to the ergodic Markov process 
with generator 
$$
\frac1 2 \Delta+ \bigg(\frac{\nabla\psi_0}{\psi_0}\bigg).\nabla
$$  
starting from $0$ with invariant measure $\mu_0$ and density $\psi_0^2$. Then the Girsanov transformation
and the ergodic theorem imply that, for $t_0$ large,
\begin{equation}\label{eqheur}
\E\bigg\{\e^{t_0 H(L_{t_0}\otimes\mu_0)} \1\{W_{t_0}\in \d y\}\bigg\} \approx \e^{\rho t_0} \psi_0(0) \psi_0(y),
\end{equation} 
where $\rho$ is the variational formula 
defined in \eqref{rhodef}. 
Summarizing the contributions in the splitting \eqref{eqheur0}, it turns out that, on the event $\{L_t\approx x\}$,
$$
H(L_t^{\ssup {\mathrm{{shift}}}})\approx H(L_t) + Ct_0.
$$
Substituting this in \eqref{ratio}, we obtain
$$
\E\bigg\{\exp\{tH(L_t)\}\,\1_{\{L_t\approx \mu_x\}}\bigg\} \,{\leq}\, \e^{-Ct_0}\, \E\bigg\{\exp\{tH(L_t)\}\,\1_{\{L_t\approx \mu_0\}}\bigg\}.$$
Recall that $t_0$ is large, since $|x|$ is large. Hence, the ratio \eqref{ratio} gets small
uniformly in large $t$. This implies that under $\widehat\P_t$, $L_t$ must have its main weight close to the starting point. This ends our survey on the proof of the tightness in Theorem~\ref{thm3}.

The additional argument for the proof of Theorem \ref{thm1} 
involves combining \eqref{eqheur} with a similar statement with the r\^{o}les of $x$ and $0$ interchanged.
Combining it with the aforementioned tightness argument then leads to the conclusion
$$
\lim_{t\to\infty} \frac {\widehat \P_t(L_t\approx \mu_x)}{\widehat \P_t(L_t\approx \mu_0)}= \frac{\psi_x(0)}{\psi_0(0)}= \frac{\psi_0(-x)}{\psi_0(0)}
$$
The above statement then implies that, 
$$
\lim_{t\to\infty} \widehat \P_t(L_t\approx \mu_x) =  \frac{\psi_0(-x)}{\int_{\R^3} \d y \, \psi_0(y) }= \frac{\psi_0(x)}{\int_{\R^3} \d y \, \psi_0(y) }.
$$
Combined with the tube property stated in Theorem \ref{thm-tubeweak}, Theorem \ref{thm1} then follows. 
Justifying the above heuristic idea will be the content of the rest of the article. 

 \section{Tightness: Proof of Theorem \ref{thm2}}\label{sectionthm2proof} 
 In this section 
 we will prove Theorem \ref{thm3}. Theorem \ref{thm4} contains the main argument. 

Let $B_r(x)$ denote the ball of radius $r>0$ around $x\in \R^3$ and
$$
\tau_{r}(x)= \inf \{s>0\colon \big|W_s- x\big| \leq r\}
$$
be the first hitting time of $B_{r}(x)$.
We also denote by $\xi _{r}(x)$ the time the Brownian path spends in $B_{1}\big(W_{\tau _{r}(x) }\big) $ after time $\tau _{r}(x)$, before exiting this ball for the first time.




Given any $\eps>0$, we choose a radius $r_\eps>0$ so that
\begin{equation}\label{radiusrequirement}
r_\eps/2\geq 1/\eps, \qquad\mu_0(B_{r_\eps/2}(0)^{\rm c})\leq \eps, \qquad \psi_0^2(\cdot)\leq \eps \,\,\mbox{on} \,\,B_{r_\eps/2}(0)^{\rm c}.
\end{equation}

We will denote by $\d$ the Prohorov metric on the set of probability measures on $\R^3$. This metric induces the weak topology, which is governed by test integrals against continuous and bounded functions. 
For any two probability measures $\mu, \nu$ on $\R^3$, we will also denote by $\|\mu-\nu\|_{\mathrm{TV}}$ the total-variation distance between $\mu$ and $\nu$.

\begin{theorem}\label{thm4}
If $\eps>0$ is chosen small enough, there exists $M(\eps)>0$ such that for $\min\{|x|, t, u \}> M(\eps)$ and any $\theta\in (0,1]$, 
\begin{equation}
\widehat \P_t\bigg\{L_t\in U_\eps(\mu_x), \xi_{r_\eps}(x)> \theta, \tau_{r_\eps}> u\bigg\}\leq \frac C\theta \e^{-u/C}
\end{equation}
for some universal constant $C>0$. Here $U_\eps(\mu)$ denotes the $\eps$-neighborhood of $\mu$ in the Prohorov metric.
\end{theorem}

\begin{proof}
We will prove this theorem in several steps. To abbreviate notation, we will write 
$$
\tau= \tau_{r_\eps}(x), \quad\xi=\xi_{r_\eps}(x), \quad B_1=B_{1}(W_\tau).
$$
Along the way, we will  introduce positive constants $c_1,c_2,c_3,\dots$, which do not depend on $t$ nor on $\eps$ nor on randomness,  whose values will not alter after
being introduced. However, there will be a universal constant $C$ whose value may and will change from appearance to appearance.  Throughout the proof, we will denote by 
\begin{equation}\label{def-A}
A_t(x,\theta,u)=A_t(x,\eps,\theta,u)=\bigg\{L_t\in U_\eps(\mu_x), \xi_{r_\eps}(x)> \theta, \tau_{r_\eps}> u\bigg\}
\end{equation}
{\bf{STEP 1:}} In this step we will prove
\begin{lemma}\label{lemma-step1}
For some constant $c_1>0$,
$$\widehat\P_t\big\{A_t(x,\theta,u)\big\}\leq \frac 1 {\theta\,Z_t}\,\,\int_u^{c_1\eps t+ 1} \d t_0\,\,\E \bigg\{\e^{tH(L_t)}\, \1_{\big\{L_t\in U_\eps(\mu_x), \,\,W_{t_0}\in {B_1}, \,\,\tau>t_0-\theta\big\}}\bigg\}.
$$
\end{lemma}
\begin{proof}
Let us first note that, on the event $\big\{L_t\in U_\eps(\mu_x)\big\}$, for some constant $c_1>0$, 
$$
\tau \leq c_1 \eps t < t.
$$
Indeed, let us choose $M(\eps)>r_\eps$. Since we are interested in the region $|x|> M(\eps)$, we have $0\notin B_{r_\eps}(x)$. Furthermore, since $|x|> r_\eps$ and 
$L_t\in U_\eps(\mu_x)$,  the time the Brownian motion spends outside $B_{r_\eps}(x)$ is less than a proportion of $\eps t$. Hence, on the event $\{L_t\in U_\eps(\mu_x)\}$, $\tau\leq c_1\eps t\leq t$, for some constant $c_1>0$.

Now, let us estimate
\begin{equation}\label{prooflemma3est00}
\begin{aligned}
&\E \bigg\{\e^{tH(L_t)}\, \1_{\big\{L_t\in U_\eps(\mu_x), \,\,\xi> \theta, \,\, \tau> u\big\}}\bigg\}
\\
&\leq\E \bigg\{\e^{tH(L_t)}\, \1_{\big\{L_t\in U_\eps(\mu_x), \,\,\xi> \theta, \,\, c_1 \eps t \geq \tau> u\big\}}\bigg\} \\
&\leq \frac 1\theta\int_u^{c_1\eps t+ \theta} \d t_0\,\,\E \bigg\{\e^{tH(L_t)}\, \1_{\big\{L_t\in U_\eps(\mu_x), \,\,\xi> \theta, \,\,t_0\geq \tau>t_0-\theta\big\}}\bigg\}
\\
&\leq \frac 1\theta\int_u^{c_1\eps t+ 1} \d t_0\,\,\E \bigg\{\e^{tH(L_t)}\, \1_{\big\{L_t\in U_\eps(\mu_x), \,\,W_{t_0}\in {B_1}, \,\,\tau>t_0-\theta\big\}}\bigg\}.
\end{aligned}
\end{equation}
In the last step we have argued that on the event $\{\xi> \theta, t_0-\theta<\tau\leq t_0\}$, the path is the on the ball $B_1$ of radius $1$ around $W_\tau$. Furthermore, we chose $\eps>0$ small enough so that $c_1\eps t+ 1\leq t$ for all sufficiently large $t$,  such that $t_0<t$ inside the integrand. Also, we will assume that $t$ is so large that $t_0\leq c_1\eps t+ 1 \leq C\eps t$ for some $C$. This proves Lemma \ref{lemma-step1}.
\end{proof}

{\bf{STEP 2:}} In this step we will make some replacements in the event appearing in the last display in \eqref{prooflemma3est00} and decompose the expectation inside the integrand accordingly.
The precise statement concerning this decomposition can be found in Lemma \ref{lemma-step2}.

For our purposes, let us fix $t_0 \in [u, c_1 \eps t+ 1]$ and write the convex combination
\begin{equation}\label{Ltsplit}
L_t= \frac {t_0} t L_{t_0}+ \frac{(t-t_0)} t L_{t_0,t}
\end{equation}
with 
$$
L_{t_0,t}= \frac 1{t-t_0} \int_{t_0}^t \delta_{W_s} \d s
$$
denoting the normalized occupation measure of the Brownian path in the time interval $[t_0,t]$.  Since, 
$$
\|L_t- L_{t_0,t}\|_{\mathrm{TV}} \leq 2 t_0/t,
$$
we clearly have $\d\big(L_t, L_{t-t_0}\big) \leq  c_2 \eps$ for some constant $c_2$ and all sufficiently small $\eps$. On the event $\{L_t\in U_\eps(\mu_x)\}$, we have $L_{t-t_0}\in U_{c_2\eps}(\mu_x)$. Moreover, let us also denote by
$$
\Lambda_{t_0,t}(y)= \int \frac{L_{t_0,t}(\d z)}{|z-y|}
$$
the Coulomb functional of $L_{t_0,t}$.

Now we are ready to make the replacements in the event appearing in the last display in \eqref{prooflemma3est00}. Then, for some constant $c_3>0$ to be chosen later, 
\begin{equation}\label{prooflemma3est1}
\begin{aligned}
&\E \bigg\{\e^{tH(L_t)}\, \1_{\big\{L_t\in U_\eps(\mu_x), \,\,W_{t_0}\in {B_1}, \,\,\tau>t_0-\theta\big\}}\bigg\} \\
&\leq \E \bigg\{\e^{tH(L_t)}\, \1_{\big\{L_{t_0,t}\in U_{c_2\eps}(\mu_x), \,\,W_{t_0}\in {B_1}, \,\,\tau>t_0-\theta, \,\, \|\Lambda_{t_0,t}-\Lambda\psi_x^2\|_\infty\leq c_3\sqrt\eps\big\}}\bigg\}\\
&\qquad\qquad+\E \bigg\{\e^{tH(L_t)}\, \1_{\big\{L_{t}\in U_{\eps}(\mu_x), \,\, \, \|\Lambda_{t_0,t}-\Lambda\psi_x^2\|_\infty> c_3\sqrt\eps\big\}}\bigg\}.
\end{aligned}
\end{equation}
The second term on the right hand side above can be rewritten as, for some constant $c_4<c_3$, 
\begin{equation}\label{prooflemma3est1.5}
\begin{aligned}
&\E \bigg\{\e^{tH(L_t)}\, \1_{\big\{L_{t}\in U_{\eps}(\mu_x), \,\, \, \|\Lambda_{t_0,t}-\Lambda\psi_x^2\|_\infty> c_3\sqrt\eps\big\}}\bigg\} \\
&=
\E \bigg\{\e^{tH(L_t)}\, \1_{\big\{L_{t}\in U_{\eps}(\mu_x), \,\, \, \|\Lambda_{t}-\Lambda\psi_x^2\|_\infty> c_4\sqrt\eps\big\}}\bigg\}  \\
&\qquad+\E \bigg\{\e^{tH(L_t)}\, \1_{\big\{\|\Lambda_{t_0,t}-\Lambda\psi_x^2\|_\infty> c_3\sqrt\eps, \,\, \, \|\Lambda_{t}-\Lambda\psi_x^2\|_\infty\leq c_4\sqrt\eps\big\}}\bigg\}
\end{aligned}
\end{equation}
Let us estimate the second term on the right hand side above. Note that 
$$
\|\Lambda_t-\Lambda_{t_0,t}\|_\infty= \frac{t_0}t \|\Lambda_{t_0}-\Lambda_{t_0,t}\|_\infty
\leq C\eps\|\Lambda_{t_0}-\Lambda_{t_0,t}\|_\infty
\leq C\eps\big(\|\Lambda_{t_0}\|_\infty+\|\Lambda_{t_0,t}\|_\infty\big).
$$
On the prescribed events on the second term on the right hand side of \eqref{prooflemma3est1.5}, 
$$
\|\Lambda_t-\Lambda_{t_0,t}\|_\infty \geq \|\Lambda_{t_0,t}-\Lambda\psi_x^2\|_\infty - \|\Lambda_t-\Lambda\psi_x^2\|_\infty \geq (c_3-c_4)\sqrt\eps.
$$
Hence, we have the estimate
\begin{equation}\label{prooflemma3est2}
\begin{aligned}
&\E \bigg\{\e^{tH(L_t)}\, \1_{\big\{\|\Lambda_{t_0,t}-\Lambda\psi_x^2\|_\infty> c_3\sqrt\eps, \,\, \, \|\Lambda_{t}-\Lambda\psi_x^2\|_\infty\leq c_4\sqrt\eps\big\}}\bigg\} \\
&\leq \E \bigg\{\e^{tH(L_t)}\, \1_{\big\{\|\Lambda_{t_0}\|_\infty \geq a_\eps\big\}}\bigg\}+ \E \bigg\{\e^{tH(L_t)}\, \1_{\big\{\|\Lambda_{t_0,t}\|_\infty \geq b_\eps\big\}}\bigg\}
\end{aligned}
\end{equation}
where $a_\eps=a \eps^{-1/2}$, $b_\eps=b \eps^{-1/2}$ for some $a, b>0$. Summarizing, we have proved 
\begin{lemma}\label{lemma-step2}
For some suitably chosen constants $c_2,c_3,c_4>0$, 
\begin{equation}\label{eq-lemma-step2}
\begin{aligned}
&\E \bigg\{\e^{tH(L_t)}\, \1_{\big\{L_t\in U_\eps(\mu_x), \,\,W_{t_0}\in {B_1}, \,\,\tau>t_0-\theta\big\}}\bigg\} \\
&\leq \E \bigg\{\e^{tH(L_t)}\, \1_{\big\{L_{t_0,t}\in U_{c_2\eps}(\mu_x), \,\,W_{t_0}\in {B_1}, \,\,\tau>t_0-\theta, \,\, \|\Lambda_{t_0,t}-\Lambda\psi_x^2\|_\infty\leq c_3\sqrt\eps\big\}}\bigg\} \\
&\qquad+\E \bigg\{\e^{tH(L_t)}\, \1_{\big\{L_{t}\in U_{\eps}(\mu_x), \,\, \, \|\Lambda_{t}-\Lambda\psi_x^2\|_\infty> c_4\sqrt\eps\big\}}\bigg\}  \\
&\qquad+\E \bigg\{\e^{tH(L_t)}\, \1_{\big\{\|\Lambda_{t_0}\|_\infty \geq a_\eps\big\}}\bigg\}+ \E \bigg\{\e^{tH(L_t)}\, \1_{\big\{\|\Lambda_{t_0,t}\|_\infty \geq b_\eps\big\}}\bigg\}
\end{aligned}
\end{equation}
where $a_\eps=a \eps^{-1/2}$, $b_\eps=b \eps^{-1/2}$ for some $a, b>0$.
\end{lemma}\qed


{\bf{STEP 3:}} Estimating the last two remainder terms on right hand side of \eqref{eq-lemma-step2} will be the task of the next step
and we will prove the following important lemma.
\begin{lemma}\label{step3lemma}
For any $a, b>0$ large enough,
\begin{equation}\label{prooflemma3est2.5}
\begin{aligned}
&\limsup_{t\to \infty} \,\,\sup_{t_0\leq  t}\,\,\frac 1 {t_0}\log \widehat\P_t\bigg\{\|\Lambda_{t_0}\|_\infty >a\bigg\} <0,\\
&\limsup_{t\to \infty} \,\, \sup_{t_0\leq  t}\,\,\frac 1 {t- t_0}\log \widehat\P_t\bigg\{\|\Lambda_{t_0,t}\|_\infty >b\bigg\} <0.
\end{aligned}
\end{equation}
\end{lemma}

\begin{proof}

Let us prove the first statement in \eqref{prooflemma3est2.5} and denote by $A_{t_0}= \{\|\Lambda_{t_0}\|_\infty >a\}$, and 
note that, for any $\sigma>0$, $H(L_\sigma)=\langle \Lambda_\sigma, L_\sigma\rangle \leq \|\Lambda_\sigma\|_\infty$. We recall the convex decomposition \eqref{Ltsplit}. Then the Hamiltonian $H(L_t)$ also decomposes
accordingly and can be estimated as 
\begin{equation}\label{Hdecomp}
\begin{aligned}
t H(L_t)&= \frac{t_0^2} t \big\langle \Lambda_{t_0}, L_{t_0}\big\rangle+ 2 \frac{t_0 (t-t_0)}t \big\langle \Lambda_{t_0}, L_{t_0,t}\big\rangle+ \frac{(t-t_0)^2}t H(L_{t_0,t}) \\
&\leq  t_0 \|\Lambda_{t_0}\|_\infty + 2t_0 \|\Lambda_{t_0}\|_\infty +\frac{(t-t_0)^2}t H(L_{t_0,t})
\end{aligned}
\end{equation}
since $t_0\leq  t$. Then, by the Markov property at time $t_0$,
$$
\begin{aligned}
\E\bigg\{\e^{tH(L_t)} \, \1_{A_{t_0}}\bigg\}
&\leq  \E\bigg\{\bigg(\e^{Ct_0\|\Lambda_{t_0}||_\infty} \, \1_{A_{t_0}}\bigg)\,\,\e^{\frac{(t-t_0)^2}t H(L_{t_0,t})}\bigg\} \\
&=\E\bigg[\bigg\{\e^{Ct_0\|\Lambda_{t_0}||_\infty} \, \1_{A_{t_0}}\bigg\}\,\ \,\E_{W_{t_0}}\bigg\{\e^{\frac{(t-t_0)^2}t H(L_{t-t_0})}\bigg\}\bigg] \\
&=\E\bigg\{\e^{Ct_0\|\Lambda_{t_0}||_\infty} \, \1_{A_{t_0}}\bigg\}\,\E\bigg\{\e^{\frac{(t-t_0)^2}t H(L_{t-t_0})}\bigg\}.
 \end{aligned}
 $$
In the last identity above we used the shift-invariance of $H$. On the other hand, 
by the above decomposition of $t H(L_t)$, we have a lower bound for the partition function, 
$$
Z_t= \E\big\{\e^{tH(L_t)}\big\} \geq  \E\big\{\e^{\frac{(t-t_0)^2}t H(L_{t_0,t})}\big\} =  \E\big\{\e^{\frac{(t-t_0)^2}t H(L_{t-t_0})}\big\}.
$$
Then, for $t_0\leq  t$, 
 $$
\widehat\P_t\big\{A_{t_0}\big\} \leq \E\big\{\e^{Ct_0\|\Lambda_{t_0}\|_\infty} \, \1_{A_{t_0}}\big\} \leq  \E\big\{\e^{2Ct_0\|\Lambda_{t_0}\|_\infty}\big\}^{1/2} \P(A_{t_0})^{1/2}.
 $$
By Proposition~\ref{expdecaysupnorm}, for any $a>0$,
$$
\limsup_{t_0\to\infty} \frac 1 {2t_0} \log \P\{A_{t_0}\}= \limsup_{t_0\to\infty} \frac 1 {2t_0} \log \P\big\{\|\Lambda_{t_0}\|_\infty >a\big\}<0.
$$
The proof of Corollary 1.4 in \cite{KM15} reveals that the above negative exponential rate 
can be made as large as needed if we chose $a>0$ large enough. Hence, the first
assertion in \eqref{prooflemma3est2.5} is proved, once we justify, for any $C>0$,
\begin{equation}\label{prooflemma3est2.55}
\limsup_{t_0\to\infty} \frac 1 {2t_0} \log \E\big\{\e^{Ct_0\|\Lambda_{t_0}\|_\infty}\big\} <\infty.
\end{equation}
This assertion will follow from the regularity properties of the random function $x\mapsto \Lambda_1(x)$, derived in in \cite{KM15}. Indeed, note that, by successive conditioning, the Markov property and the shift-invariance of $\|\Lambda_1\|_\infty$,
it is enough to justify that some exponential moment of $\|\Lambda_1\|_\infty$ is finite. Note that, we can write, for any $\delta>0$, 
\begin{equation}\label{prooflemma3est2.56}
\|\Lambda_1\|_\infty \leq \sup_{x_1, x_2 \in \R^3\colon |x_1-x_2|\leq\delta} \big|\Lambda_1(x_1)-\Lambda_1(x_2)\big| + \sup_{x\in \delta\Z^3} \int_0^1 \frac {\d s} {|W_s -x|}.
\end{equation}
Let us now handle the first summand. In Lemma 2.2 in \cite{KM15} we proved that, for any $\eta\in (\frac 13, \frac 12)$, if 
$$
M= \int\int_{|x_1-x_2|\leq 1}\d x_1\d x_2\, \bigg[\exp\bigg\{\beta\bigg(\frac{|\Lambda_1(x_1)-\Lambda_1(x_2)|}{|x_1-x_2|^a}\bigg)^\rho\bigg\}- 1 \bigg],
$$
where $\rho=\frac 1 {1-\eta}>1$ and $a=1-2\eta>0$, then, for some $\beta>0$,  
\begin{equation}\label{prooflemma3est2.57}
\E(M) <\infty.
\end{equation}
Using the Garsia-Rodemich-Rumsey estimate, we also proved that (see the proof of Proposition 1.3 in \cite{KM15}), for some fixed constant $\gamma>0$, 
$$
\sup_{|x_1-x_2|\leq \delta} \big|\Lambda_1(x_1)-\Lambda_1(x_2)\big| \leq \frac{1-2\eta}{\beta^{1/\rho}} \int_0^\delta \log\bigg(1+ \frac M{\gamma u^6}\bigg)^{1/\rho} \, u^{-2\eps} \d u.
$$
Now if we choose $\delta$ small enough, then the right hand side above is smaller than 
$$
\frac{1-2\eta}{\beta^{1/\rho}} C(\delta) \log\big(M\vee 1\big)^{1/\rho}
$$
for some constant $C(\delta)$ which goes to $0$ as $\delta\to 0$. Hence, for any $C>0$, by \eqref{prooflemma3est2.57}, we have
$$
\E\bigg\{\e^{C\sup_{ |x_1-x_2|\leq\delta} \big|\Lambda_1(x_1)-\Lambda_1(x_2)\big|}\bigg\}<\infty.
$$
Let us turn to the second term on the right hand side of \eqref{prooflemma3est2.56}. Since we are interested in the behavior 
of the path in the time horizon $[0,1]$, it is enough to estimate the supremum in a bounded box. We will show that, for any fixed $\delta>0$ and any $C>0$,
\begin{equation}\label{prooflemma3est2.58}
\E\bigg[\sup_{\heap{x\in \delta\Z^3}{|x\leq 2}}\exp\bigg\{ C\int_0^1 \frac {\d s} {|W_s -x|}\bigg\}\bigg] \leq (2/\delta)^3\E\bigg[\exp\bigg\{ C\int_0^1 \frac {\d s} {|W_s |}\bigg\}\bigg] <\infty.
\end{equation}
For any $\eta>0$, we can write $1/|x|= V_\eta(x)+ Y_\eta(x)$ for $V_\eta(x)= 1/(|x|^2+\eta^2)^{1/2}$. Since, for any fixed $\eta>0$, $V_\eta$ is a bounded function, 
the above claim holds with $V_\eta(W_s)$ replacing $1/|W_s|$. Hence, (by Cauchy-Schwarz inequality, for instance), it suffices to check the above statement with the difference $Y_\eta(W_s)$, which can be written as
$$
\begin{aligned}Y_\eta(x)=\frac{1}{|x|}-\frac{1}{\sqrt {\eta^2+|x|^2}}
&=\frac{\sqrt {\eta^2+|x|^2}-|x|}{|x|\sqrt{\eps^2+|x|^2}}=\frac{\eta^2}{|x|+\sqrt {\eta^2+|x|^2}}\,\,\frac{1}{\sqrt {\eta^2+|x|^2}}\,\,\frac{1}{|x|}\\
&=\eta^{-1}\phi\bigg(\frac{x}{\eta}\bigg),
\end{aligned}
$$
with
$$
\phi(x)=\frac{1}{|x|}\,\,\frac 1{\sqrt{1+|x|^2}}\,\, \frac 1{|x|+\sqrt{1+|x|^2}}.
$$
One can bound  $\phi(x)$ by $\frac{b}{|x|^\frac{3}{2}}$, since it behaves like $\frac{1}{|x|}$ near $0$ and like $\frac{1}{|x|^3}$ near $\infty$. In particular 
$$
Y_\eta(x)\le \frac{b\sqrt{\eta}}{|x|^\frac{3}{2}}.
$$
Hence, for \eqref{prooflemma3est2.58}, it suffices to show, for $\eta>0$ small enough and any $C>0$,
\begin{equation}\label{prooflemma3est2.59}
\E\bigg[\exp\bigg\{C b \sqrt\eta \int_0^1 \frac{\d s}{|W_s|^{3/2}}\bigg\}\bigg] <\infty.
\end{equation}
For this, we appeal to Portenko's lemma (see \cite{P76}), which states that, if for a Markov process $\{\P^{\ssup x}\}$ and for a function $\widetilde V\ge 0$
$$
\sup_{x\in \R^3} \E^{\ssup x}\bigg\{\int_0^1 \widetilde V(W_s)\d s\bigg\}\leq \gamma<1
$$
then
$$
\sup_{x\in \R^3} \E^{\ssup x}\bigg\{\exp\bigg\{\int_0^1 \widetilde V(W_s)\d s\bigg\}\bigg\}\le \frac{\gamma}{1-\gamma} <\infty.
$$
Hence, to prove \eqref{prooflemma3est2.59}, we need to verify that 
$$
\begin{aligned}
\sup_{x\in \R^3} \E^{\ssup x}\bigg\{\int_0^1 \frac{\d \sigma}{|W_\sigma|^\frac{3}{2}}\bigg\}
=\sup_{x\in \R^3} \int_0^1 \d \sigma\int_{\R^3} \d y \,\, \frac{1}{|y|^\frac{3}{2}} \frac{1}{(2\pi \sigma)^\frac{3}{2} }\exp\bigg\{-\frac{(y-x)^2}{2\sigma}\bigg\}
<\infty.
\end{aligned}
$$
One can see that  
$$
\sup_{x\in \R^3} \int_{\R^3}\d y\frac{1}{|y|^\frac{3}{2}} \frac{1}{(2\pi \sigma)^\frac{3}{2} }\exp\bigg\{-\frac{(y-x)^2}{2\sigma}\bigg\}
$$ 
is attained at $x=0$ because we can rewrite the integral by Parseval's identity as 
$$ c\int_{\R^3} \exp\bigg\{-\frac{\sigma|\xi|^2}{2}+i\langle x,\xi\rangle\bigg\}\frac{1}{|\xi|^\frac{3}{2}}d\xi,
$$ 
where $c>0$ is a constant. When $x=0$, the integral reduces to $\int_0^1 \sigma^{-3/4} \,\,\d \sigma$, which is finite. This finishes
the proof of the first assertion in \eqref{prooflemma3est2.5}. The second assertion follows essentially the same arguments, if we upper estimate by Markov property,
$$
\E\bigg\{\e^{tH(L_t)} \, \1_{\{\|\Lambda_{t_0,t}\|_\infty >b\eps^{-1/2}\}}\bigg\} \leq \E\bigg\{\e^{\frac{t_0^2}t H(L_{t_0})}\bigg\} \E\bigg\{ \e^{C(t-t_0) \|\Lambda_{t-t_0}\|_\infty} \1_{\{\|\Lambda_{t-t_0}\|_\infty >b\eps^{-1/2}\}}\bigg\}
$$
and lower estimate
$$
Z_t=\E\big\{\e^{tH(L_t)}\big\} \geq \E\big\{\e^{\frac{t_0^2}t H(L_{t_0})}\big\}.
$$
Lemma \ref{step3lemma} is proved.
\end{proof}



{\bf{STEP 4:}} In this step we will estimate the second term on the right hand side in \eqref{eq-lemma-step2}. We will show that this term is also negligible as $t\to\infty$ by proving the following 
\begin{lemma}\label{lemma-step4}
Uniformly in $x$ on compacts (in particular, for the $x$ chosen in the statement of Theorem \ref{thm4}) and for small enough $\eps>0$,
\begin{equation}\label{prooflemma3est5}
\limsup_{t\to\infty}\frac 1t \log \widehat\P_t\bigg\{L_{t}\in U_{\eps}(\mu_x), \,\, \, \|\Lambda_{t}-\Lambda\psi_x^2\|_\infty> c_4\sqrt\eps\bigg\}<0.
\end{equation}
\end{lemma}
\begin{proof} By Theorem \ref{tubeunifmetric}, it is enough to justify,
\begin{equation}\label{prooflemma3est6}
\begin{aligned}
&\limsup_{t\to\infty}\frac 1t \log \widehat\P_t\bigg\{L_{t}\in U_{\eps}(\mu_x), \,\, \, \|\Lambda_{t}-\Lambda\psi_x^2\|_\infty> c_4\sqrt\eps,\\
&\qquad\qquad\qquad\qquad \qquad\qquad\inf_{y\in \R^3} \|\Lambda_{t}-\Lambda\psi_y^2\|_\infty\leq c_4\sqrt\eps\bigg\}<0.
\end{aligned}
\end{equation}
Suppose $y\in \R^3$ be such that $|y-x|\leq \sqrt\eps$. Since $\psi_0^2$ is a smooth function vanishing at infinity (see \cite{L76}) and the Coulomb function $x\mapsto 1/|x| $ lies in $L^1_{\mathrm{loc}}(\R^3)$, the function
$\Lambda\psi^2_0= \psi_0^2\star1/|\cdot|$ is smooth and hence a Lipschitz function. Hence, for some $c_5$, we have
$$
\big|\Lambda\psi_x^2(x)- \Lambda\psi_y^2(x)\big|= \big|\Lambda\psi_0^2(0)- \Lambda\psi_0^2(y-x)\big| \leq c_5 \sqrt \eps.
$$
Then, if we chose $c_4>c_5$, on the event, $\|\Lambda_{t}-\Lambda\psi_x^2\|_\infty> c_4\sqrt\eps$, for $|y-x|\leq \eps$,
\begin{equation}\label{prooflemma3est6.5}
 \|\Lambda_{t}-\Lambda\psi_y^2\|_\infty \geq (c_4-c_5) \sqrt \eps.
 \end{equation}
On the other hand, since $\psi_0^2$ is concentrated at $0$, a simple argument using polar coordinates and triangle inequality shows that, 
if $|y-x|\geq \sqrt\eps$,  
 $$
\big|\Lambda\psi_x^2(x)- \Lambda\psi_y^2(x)\big|= \big|\Lambda\psi_0^2(0)- \Lambda\psi_0^2(y-x)\big| \geq c_6 \sqrt \eps.
$$ 
Hence, for $|y-x|\geq \eps$ and for any $\eta>0$,
$$
\begin{aligned}
\|\Lambda_{t}-\Lambda\psi_y^2\|_\infty&= \sup_{w\in\R^3}\big| \Lambda_{t}(w)-\Lambda\psi_y^2(w)\big| \\
&\geq \big|\Lambda\psi_x^2(x)- \Lambda\psi_y^2(x)\big|- \bigg| \int_{B_\eta(x)} \frac{L_t(\d z)- \psi_x^2(z)\d z}{|z-x|} +\int_{B_\eta(x)^c} \frac{L_t(\d z)- \psi_x^2(z)\d z}{|z-x|}\bigg| \\
&\geq c_6 \sqrt\eps- \bigg[\int_{B_\eta(x)} \frac{L_t(\d z)}{|z-x|} + \int_{B_\eta(x)}\frac{\psi_x^2(z)\d z}{|z-x|}+ \frac 1 \eta\bigg\langle \frac{\eta}{|\cdot- x|}\wedge 1, L_t- \psi_x^2\bigg\rangle\bigg]\\
&\geq c_6 \sqrt\eps- \bigg[\int_{B_\eta(0)} \frac{L_t(\d z)}{|z|} + \int_{B_\eta(0)}\frac{\psi_0^2(z)\d z}{|z-x|}+ \frac 1 \eta\d\big(L_t,\psi_x^2\big)\bigg]\\
&\geq c_6 \sqrt\eps- \bigg[\int_{B_\eta(0)} \frac{L_t(\d z)}{|z|} + \eta^2+ \frac \eps\eta\bigg],
\end{aligned}
$$
since $L_t\in U_\eps(\psi_x^2)$. Let us chose $\eta=\sqrt \eps$. Then above estimate and \eqref{prooflemma3est6.5} imply that, for \eqref{prooflemma3est6}, it is enough to derive, for some constant $c_7>0$
and $\eps$ small enough, 
$$
\limsup_{t\to\infty}\frac 1t \log \widehat\P_t\bigg\{\int_{B_{\sqrt\eps}(0)} \frac {L_t(\d z)}{|z|} > c_7 \sqrt\eps\bigg\}<0.
$$
But the above fact follows from \cite{KM15} (see the proof of Eq.(3.6), p.15, \cite{KM15}). Hence, \eqref{prooflemma3est6} and Lemma \ref{lemma-step4} is proved. 
\end{proof}

{\bf{STEP 5: Proof of Theorem \ref{thm4}.}} 

In this step we will prove Theorem \ref{thm4}. Recall the requirement \eqref{radiusrequirement}, the event $A_t(x,\theta,u)$ from \eqref{def-A} and that $\tau=\tau_{r_\eps}(x)$ denotes the first hitting time of the ball $B_{r_\eps}(x)$, and $\xi=\xi_{r_\eps}(x)$
stands for the time the path spends in the ball $B_1(W_\tau)$, after time $\tau$ and before exiting $B_1(W_\tau)$ for the first time. Note that we need to show, 
that for $\eps>0$  small enough, there exists $M(\eps)>0$ such that for $\min\{|x|, t, u \}> M(\eps)$ and any $\theta\in (0,1]$, 
\begin{equation}
\widehat \P_t\big\{A_t(x,u,\theta)\big\}\leq \frac C\theta \e^{-u/C}
\end{equation}
for some universal constant $C>0$.

We note that in order to prove the above estimate, thanks to Lemma \ref{lemma-step1}, Lemma \ref{lemma-step2}, Lemma \ref{step3lemma} and Lemma \ref{lemma-step4}, 
it is enough to estimate the first term on the right hand side of \eqref{prooflemma3est1}. 
Note that we can rewrite this term as 
\begin{equation}\label{prooflemma3est6.5}
\begin{aligned}
&\E \bigg\{\e^{tH(L_t)}\, \1_{\big\{L_{t_0,t}\in W_{c_2\eps}(\mu_x), \,\,W_{t_0}\in {B_1}, \,\,\tau_{}>t_0-\theta, \,\, \|\Lambda_{t_0,t}-\Lambda\psi_x^2\|_\infty\leq c_3\sqrt\eps\big\}}\bigg\}
\\
&\leq \E \bigg\{\e^{tH(L_t)}\, \1_{\big\{W_{t_0}\in {B_1}, \,\,\tau_{}>t_0-\theta, \,\, \|\Lambda_{t_0,t}-\Lambda\psi_x^2\|_\infty\leq c_3\sqrt\eps, \, \|\Lambda_{t_0}\|_\infty\leq c_8\eps^{-1/2}\big\}}\bigg\}
\\
&\qquad\qquad+ \E \bigg\{\e^{tH(L_t)}\1\{\|\Lambda_{t_0}\|_\infty> c_8\eps^{-1/2}\}\bigg\}.
\end{aligned}
\end{equation}
Again the argument of Step 3 implies that, for $t$ chosen large enough, the second summand above is exponentially small in $t_0$ if $\eps>0$ is chosen small enough and 
$t_0$ is chosen large enough, see \eqref{prooflemma3est2.5}. Hence, we turn to the first summand on the right hand side in \eqref{prooflemma3est6.5}
and show that, for some constant $c_9>0$,

\begin{equation}\label{eq1-step5}
\begin{aligned}
&\widehat\P_t\bigg\{L_{t_0,t}\in W_{c_2\eps}(\mu_x), \,\,W_{t_0}\in {B_1}, \,\,\tau_{}>t_0-\theta, \,\, \|\Lambda_{t_0,t}-\Lambda\psi_x^2\|_\infty\leq c_3\sqrt\eps\bigg\} \\
&\leq \e^{C\sqrt\eps t_0} \,\,\e^{-c_9 t_0}.
\end{aligned}
\end{equation}
Combined with this estimate, as remarked before, Lemma \ref{lemma-step1}-Lemma \ref{lemma-step4} imply that, 
$$
\begin{aligned}
\widehat\P_t\big\{A_t(x,u,\theta)\big\}=\widehat\P_t \bigg\{ L_t\in U_\eps(\mu_x), \xi> \theta, \tau> u\bigg\}
&\leq \frac 1\theta\int_u^{c_1\eps t+ 1} \d t_0\,\,\e^{C\sqrt\eps t_0}\, \e^{-c_9t_0} \\
&\leq \frac C{\theta}\e^{-u/C},
\end{aligned}
$$
as desired in Theorem \ref{thm4}.

It remains to prove \eqref{eq1-step5}. Let us rewrite the decomposition of the Hamiltonian as
\begin{equation}\label{splitH}
t H(L_t)= \frac{t_0^2} t \big\langle \Lambda_{t_0}, L_{t_0}\big\rangle+ 2 \frac{t_0 (t-t_0)}t \big\langle \Lambda_{t_0}, L_{t_0,t}\big\rangle+ \frac{(t-t_0)^2}t H(L_{t_0,t}).
\end{equation}
We will handle the three contributions on the right-hand side separately. The first term is relatively easy to handle. Recall that we assumed that $t_0 \leq C \eps t$. Hence
$\frac{t_0^2}t \leq C \eps t_0$. Hence, 
\begin{equation}\label{prooflemma3est6.55}
\begin{aligned}
\hspace{5mm}\frac {t_0^2}t H(L_{t_0}) \leq C\eps t_0 H(L_{t_0})= C\eps t_0 \big\langle \Lambda_{t_0}, L_{t_0}\big\rangle &\leq C\eps t_0 \|\Lambda_{t_0}\|_\infty \\
&\leq C\sqrt \eps t_0,
\end{aligned}
\end{equation}
on the event $\{\|\Lambda_{t_0}\|_\infty\leq c_8\eps^{-1/2}\}$ under interest.

The second term is estimated
on the prescribed event $\big\{\|\Lambda_{t_0,t}-\Lambda\psi_x^2\|_\infty \leq c_3 \eps, \tau>t_0- \theta\big\}$ as follows:
\begin{equation}\label{prooflemma3est7}
\begin{aligned}
2\frac{t_0(t-t_0)}{t}\big\langle \Lambda_{t_0,t}, L_{t_0}\big\rangle 
&\leq t_0 C \sqrt\eps + \,\,\,t_0 \big\langle \Lambda\psi_x^2, L_{t_0}\big\rangle \\
&\leq  t_0 C \sqrt\eps + t_0 \big\langle \Lambda\psi_x^2, L_{t_0-\theta}\big\rangle\\
&\qquad+ \big\langle\Lambda\psi_x^2,\big\|L_{t_0}- L_{t_0-\theta}\big\|_{\mathrm{TV}} \big\rangle \\
&\leq  t_0 C \sqrt\eps + t_0 \big\langle \Lambda\psi_x^2, L_{t_0-\theta}\big\rangle.
\end{aligned}
\end{equation}


In the last line we used the fact that 
$$
\big\|L_{t_0}- L_{t_0-\theta}\big\|_{\mathrm{TV}} \leq \frac {2\theta} {t_0} < \frac 2u \leq C\sqrt\eps.
$$

Let us now estimate $\langle \Lambda\psi_x^2, L_{t_0-\theta}\rangle$ on the event $\{\tau> t_0- \theta\}$, for which we want to use 
the fact that $\psi_x^2$ puts most of its mass around $B_{r_\eps/2}(x)$, while on the event under interest, the Brownian path until time $t_0-\theta$
has not yet touched $B_{r_\eps}(x)$, recall the requirement \eqref{radiusrequirement}. Then, 
\begin{equation}\label{prooflemma3est8}
\begin{aligned}
&\big\langle \Lambda\psi_x^2, L_{t_0-\theta}\big\rangle \,\,\, \1_{\{\tau> t_0- \theta\}} \\
& =\int_{B_{r_\eps/2}(x)} L_{t_0-\theta}(\d z)\int_{B_{r_\eps/2}(x)^c} \frac{\psi_x^2(y) L_{t_0-\theta}}{|y-z|} \, \d y \\
&\qquad\qquad+ \int_{B_{r_\eps/2}(x)^c} L_{t_0-\theta}(\d z) \int_{B_{r_\eps/2}(x)^c \cap B_1(z)} \frac{\psi_x^2(y)}{|y-z|} \, \d y \\
&\qquad\qquad+ \int_{B_{r_\eps/2}(x)^c} L_{t_0-\theta}(\d z) \int_{B_{r_\eps/2}(x)^c \cap B_1(z)^c} \frac{\psi_x^2(y)}{|y-z|} \, \d y \\
&\leq \frac 2 {r_\eps}+ \eps \int L_{t_0-\theta}(\d z)\int_{B_1(0)} \frac{\d y} {|y|}+ \eps \int L_{t_0-\theta}(\d z) \\
&\leq C\eps,
\end{aligned}
\end{equation}
where we used that $|y-z| \geq r_\eps/2$ in the first integral, while $\psi_x^2(\cdot) \leq\eps$ on $B_{r_\eps/2}(x)^c$ for the second and third integral and that $|y-z| \geq 1$ in the third integral. Furthermore, recall that
$2/r_\eps \leq \eps$. If we combine \eqref{prooflemma3est7}, \eqref{prooflemma3est8} and \eqref{splitH},  we have an estimate for the first term on the right hand side of \eqref{prooflemma3est6.5}:

\begin{equation}\label{prooflemma3est8}
\begin{aligned}
&\E \bigg\{\e^{tH(L_t)}.\, \1_{\big\{W_{t_0}\in {B_1}, \,\,\tau_{}>t_0-\theta, \,\, \|\Lambda_{t_0,t}-\Lambda\psi_x^2\|_\infty\leq c_3\sqrt\eps, \|\Lambda_{t_0}\|_\infty\leq c_8\eps^{-1/2}\big\}}\bigg\} \\
&\leq \e^{C\sqrt\eps t_0} \,\,\E \bigg\{ \e^{\frac{(t-t_0)^2}t H(L_{t_0,t})} \,\,\, \1_{\|\Lambda_{t_0,t}-\Lambda\psi_x^2\|_\infty\leq c_3\sqrt\eps \,\,W_{t_0}\in {B_1} \big\}}\bigg\}.
\end{aligned}
\end{equation}

Let us denote by $\mathcal F_{t_0,t}$ the canonical $\sigma$-field generated by $(W_s)_{s\in[t_0,t]}$. We need the following estimate to conclude the proof of 
\eqref{eq1-step5}.
\begin{lemma}\label{lemma-claim}
On the event $\{\|\Lambda_{t_0,t}-\Lambda\psi_x^2\|_\infty\leq c_3\sqrt\eps, W_{t_0}\in {B_1}\}$, we have
\begin{equation}\label{claim}
\E_x\bigg[\e^{tH(L_t)}\big|\mathcal F_{t_0,t}\bigg]\geq \exp\bigg\{c_9t_0 + \frac{(t-t_0)^2}tH(L_{t_0,t})\bigg\}.
\end{equation}
\end{lemma}
We defer the proof of Lemma \ref{lemma-claim} until Step 6 and finish the proof of Theorem \ref{thm4} based on the above estimate. Note that,
$$
\begin{aligned}
Z_t&=\E_x\bigg\{\E_x\big(\e^{tH(L_t)}\big|\mathcal F_{t_0,t}\big)\bigg\}\\
&\geq \int_{B_1} \E_x\bigg\{\E_x\big\{\e^{tH(L_t)}\big|\mathcal F_{t_0,t}\big\} \, \1\big\{\|\Lambda_{t_0,t}-\Lambda\psi_x^2\|_\infty\leq c_3\sqrt\eps, W_{t_0}\in \d y\big\}\bigg\} 
\\
&\geq \int_{B_1} \, \frac{p_{t_0}(x,y)}{p_{t_0}(y,0)}\,\E_0\bigg\{\E_x\big\{\e^{tH(L_t)}\big|\mathcal F_{t_0,t}\big\}; \, \1\big\{\|\Lambda_{t_0,t}-\Lambda\psi_x^2\|_\infty\leq c_3\sqrt\eps, W_{t_0}\in \d y\big\}\bigg\}
\end{aligned}
$$
Note that, if $M(\eps)$ is chosen large enough, then for $y\in B_1=B_1(W_\tau)$, $t\geq u> M(\eps)$ and $|x|>M(\eps)$, we have $p_{t_0}(x,y)\geq p_{t_0}(y,0)$. Then, 
$$
\begin{aligned}
Z_t&\geq 
 \int_{B_1} \E_0\bigg\{\E_x\big\{\e^{tH(L_t)}\big|\mathcal F_{t_0,t}\big\} \, \1\big\{\|\Lambda_{t_0,t}-\Lambda\psi_x^2\|_\infty\leq c\eps, W_{t_0}\in \d y\big\}\bigg\} \\
&= \e^{c_9t_0}\E_0\bigg\{\exp\bigg\{\frac{(t-t_0)^2}tH(L_{t_0,t})\bigg\}\, \1\big\{\|\Lambda_{t_0,t}-\Lambda\psi_x^2\|_\infty\leq c\eps, W_{t_0}\in {B_1} \big\}\bigg\}
\end{aligned}
$$
This finishes the proof of Theorem \ref{thm4}, assuming Lemma \ref{lemma-claim}.

{\bf{STEP 6:}} In this step we will prove Lemma \ref{lemma-claim}. 

{\bf{Proof of Lemma \ref{lemma-claim}:}} Let us again recall the splitting introduced in \eqref{splitH}. Then, on the event $\|\Lambda_{t_0,t}-\Lambda\psi_x^2\|_\infty\leq C\sqrt\eps$, 
we have a lower bound
\begin{equation}\label{HamLB}
\begin{aligned}
tH(L_t)
&\geq \frac {(t-t_0)^2}{t} H(L_{t_0,t}) + 2t_0 \frac{(t-t_0)}t \big\langle \Lambda_{t_0,t}, L_{t_0}\big\rangle\\
& \geq \frac {(t-t_0)^2}{t} H(L_{t_0,t}) + 2t_0 \frac{(t-t_0)}t\bigg[\big\langle\Lambda\psi_x^2,L_{t_0}\big\rangle- \big\langle L_{t_0},\Lambda_{t_0,t}-\Lambda\psi_x^2\big\rangle\bigg] \\
&\geq  \frac {(t-t_0)^2}{t} H(L_{t_0,t}) + 2t_0 \frac{(t-t_0)}t\bigg[\big\langle\Lambda\psi_x^2,L_{t_0}\big\rangle- \big\|\Lambda_{t_0,t}-\Lambda\psi_x^2\big\|_\infty\bigg]\\
&\geq \frac {(t-t_0)^2}{t} H(L_{t_0,t}) + 2t_0 \bigg[\big\langle\Lambda\psi_x^2,L_{t_0}\big\rangle- C\sqrt\eps\bigg]\\
&\qquad\qquad\qquad\qquad- 2\frac{t_0^2}t \,\bigg[\|\Lambda\psi_x^2\|_\infty +C\sqrt\eps\bigg]\\
&\geq \frac {(t-t_0)^2}{t} H(L_{t_0,t}) + 2t_0 \bigg[\big\langle\Lambda\psi_x^2,L_{t_0}\big\rangle- C\sqrt\eps\bigg]- \widetilde C\eps t_0.
\end{aligned}
\end{equation}
 Hence, we infer, on $\{\|\Lambda_{t_0,t}-\Lambda\psi_x^2\|_\infty\leq c\sqrt\eps, W_{t_0}\in {B_1}\}$,
\begin{equation}\label{prooflemma3est9}
\begin{aligned}
\E_x\bigg[\e^{tH(L_t)}\big|\mathcal F_{t_0,t}\bigg] 
&\geq e^{-C\sqrt\eps t_0}\, \exp\bigg\{\frac {(t-t_0)^2}{t} H(L_{t_0,t})\bigg\} \\
&\qquad\qquad \E_x\bigg\{\exp\big\{2 t_0H(L_{t_0}\otimes \mu_x)\big\}\big| \,W_{t_0}\in {B_1}\bigg\}.
\end{aligned}
\end{equation}
Now we only need to handle the expectation on the right hand side.

We consider a diffusion with generator 
$$
\mathfrak L^{\ssup{\psi_x}}= \frac 12\Delta + \bigg(\frac{\nabla \psi_x}{\psi_x}\bigg)\cdot\nabla 
$$
corresponding to an ergodic Markov process $\P_x^{\ssup {\psi_x}}$ starting from $x$ with invariant density $\psi_x^2(\cdot)$. From the underlying expectation $\E_x$ on the right hand side 
of \eqref{prooflemma3est9}, we want to switch to the corresponding expectation $\E_x^{\ssup {\psi_x}}$. By the Cameron-Martin-Girsanov formula (\cite{SV79}),
\begin{equation}\label{measchange}
\begin{aligned}
&\frac{\d \P_x}{\d \P_x^{\ssup {\psi_x}}}\,\big(\omega\big)\bigg|_{\mathcal F_{t_0}}\\
&= \exp\bigg[-\int_0^{t_0} \frac {\nabla \psi_x(\omega_s)}{\psi_x(\omega_s)}\d W_s + \, \frac 12\int_0^{t_0} \bigg|\frac{\nabla \psi_x(\omega_s)}{\psi_x(\omega_s)}\bigg|^2\, \d s\bigg]\\
&= \exp\bigg[ \log \psi_x(\omega_0)- \log \psi_x(\omega_{t_0}) +\frac 12 \int_0^{t_0} \frac {\Delta \psi_x(\omega_s)}{\psi_x(\omega_s)}\, \d s\bigg]\\
&= \frac{\psi_x(x)}{\psi_x(\omega_{t_0})} \exp\bigg[ \frac 12 \int_0^{t_0} \frac {\Delta \psi_x(\omega_s)}{\psi_x(\omega_s)}\, \d s\bigg].
\end{aligned}
\end{equation}
Let us recall the variational formula \eqref{rhodef}:
$$
\rho= \sup_{\heap{\psi\in H^1(\R^3)}{\|\psi\|_2=1}}\bigg\{\int\int_{\R^3\times\R^3} \frac{\psi^2(x)\psi^2(y)} {|x-y|} \, \d x \, \d y\,\, - \, \frac 12 \int_{\R^3} \d x |\nabla\psi(x)|^2 \bigg\}
$$
A simple perturbation argument shows that the maximizing function $\psi_0\in H^1(\R^3)$ satisfies the Euler-Lagrange equation 
\begin{equation}\label{EL}
\bigg(\Delta + 4 \int_{\R^3}\frac {\psi_0^2(y)}{|x-y|}\, \d y\bigg)\psi_0(x)= \lambda \psi_0(x).
\end{equation}
We multiply \eqref{EL} on both sides by $\psi_0(x)$, integrate over $\R^3$ and recall that $\int_{\R^3}\psi_0^2=1$, to see that
$$
\lambda=4 \int\int_{\R^3\times\R^3}\frac{\psi_0^2(x)\psi_0^2(y)}{|x-y|}-  \|\nabla\psi_0\|_2^2 \geq 2 \rho >0.
$$
Now we divide \eqref{EL} by $\psi_0(x)$, plug in $x= W(s)$ and integrate on the time interval $[0,t_0]$ to get
$$
\begin{aligned}
\int_0^{t_0} \frac {\Delta \psi_0(W_s)}{\psi_0(W_s)}\, \d s +4 \int_0^{t_0} \int_{\R^3} \frac{\psi_0^2(y)\d y} {|W_s- y|}
&=\int_0^{t_0} \frac {\Delta \psi_0(W_s)}{\psi_0(W_s)}\, \d s + 4 t_0 H(L_{t_0}\otimes \mu_0)\\
& = \lambda t_0.
\end{aligned}
$$
Repeating the same argument for $\psi_x ^2=\psi_0^2\star\delta_x$ we get 
\begin{equation}\label{prooflemma3est10}
 2 t_0 H(L_{t_0}\otimes \mu_x) = \frac{\lambda t_0}2 - \frac 12 \int_0^{t_0} \frac {\Delta \psi_x(W_s)}{\psi_x(W_s)}\, \d s.
 \end{equation}
We now perform a change of measure in the expectation of the right hand side of \eqref{prooflemma3est9}, and combine the above identity with \eqref{measchange} to get
$$
\begin{aligned}
\E_x\bigg\{\exp\big\{2t_0H(L_{t_0}\otimes \mu_x)\big\}\big| \,W_{t_0}\in {B_1}\bigg\} 
&\geq \int_{B_1} \E_x\bigg[\exp\big\{2t_0 H(L_{t_0}\otimes\mu_x)\big\}\, \1\big\{W_{t_0}\in \d y\big\}\bigg]\\
&= \e^{\lambda t_0/2}\,\int_{B_1}\,E_x^{\ssup {\psi_x}}\bigg[\frac{\psi_x(x)}{\psi_x(y)}\, \1{\big\{W_{t_0}\in \d y\big\}}\bigg] \\
&= \e^{\lambda t_0/2}\,\int_{B_1}\,\frac{\psi_0(0)}{\psi_0(y-x)}\,\P_x^{\ssup {\psi_x}}\big\{W_{t_0}\in \d y\big\}.
\end{aligned}
$$
Recall that $\P_x^{\ssup {\psi_x}}$ is ergodic with invariant measure $\mu_x(\d y)=\psi^2_x(y)\,\d y= \psi_0^2(y-x) \d y$. Hence, by the ergodic theorem,
\begin{equation}\label{prooflemma3est10.5}
\begin{aligned}
&\liminf_{t_0\to\infty} \e^{-\lambda t_0/2} \, \E_x\bigg\{\exp\big\{2t_0H(L_{t_0}\otimes \mu_x)\big\}\big| \,W_{t_0}\in {B_1}\bigg\} \\
&\geq \int_{B_1}\frac{\psi_0(0)}{\psi_0(y-x)}\, \psi_0^2(y-x)\d y 
=\psi_0(0)\,  \int_{B_1} \psi_0(y-x)\d y.
\end{aligned}
\end{equation}
We choose, $\sigma_0(\eps)$ such that, for $t_0\geq \sigma_0(\eps)$,
$$
\E_x\bigg\{\exp\big\{t_0H(L_{t_0}\otimes \mu_x)\big\}\big| \,W_{t_0}\in {B_1}\bigg\}\geq \e^{\lambda t_0/3}.
$$
Since we are interested in the regime $|x|> M(\eps)$, we need to pick $M(\eps)\geq \sigma_0(\eps)$. 
We combine this estimate with \eqref{prooflemma3est9} to prove \eqref{claim}. Theorem \ref{thm4} is proved.
\end{proof}

To finish the proof of Theorem \ref{thm3}, we need two more technical estimates. 

Recall that for any $r>0$, $\tau=\tau_r(x)$ denotes the first hitting time of the ball $B_r(x)$ and 
$\xi=\xi _{r}(x)$ denotes the time the Brownian path spends in $B_{1}\left( W_{\tau _{r}\left( x\right) }\right)$ after time $\tau _{r}(x)$, before exiting this ball for the first time. 
It is well known that, for any $\theta>0$, %
\begin{equation}\label{lemma3.3pf-1}
\P\bigg\{ \xi \leq \theta \bigg\} \leq C\exp \bigg\{ -\frac 1{ C\theta}\bigg\}
\end{equation}
for some $C>0$.

\begin{lemma}\label{onboundary} Uniformly in $t>0,r>0,x\in \mathbb{R}^{3}$, 
$$
\lim_{\theta \to 0} \widehat\P_t \bigg\{ \xi _{r}(x)\leq \theta
,\ \tau _{r}(x)\leq t\bigg\} =0.
$$
\end{lemma}

\begin{proof}
We now split at two time horizons $\tau$ and $\tau+\xi$:%
\begin{equation}\label{lemma3.3pf0}
L_{t}=\frac{\tau }{t}L_{\tau }+\frac{\xi }{t}L_{\tau ,\tau +\xi }+\frac{%
t-\tau -\xi }{t}L_{\tau +\xi ,t}.
\end{equation}
This also leads to a similar decomposition of $\Lambda_t$. 

We also write
$$
\begin{aligned}
&L_{t}^{\prime } =\frac{\tau }{t}L_{\tau }+\frac{t-\tau -\xi }{t}L_{\tau
+\xi ,t}, \\
&\Lambda _{t}^{\prime } =\frac{\tau }{t}\Lambda _{\tau }+\frac{t-\tau -\xi 
}{t}\Lambda _{\tau +\xi ,t}
\end{aligned}
$$
and
$$
\begin{aligned}
&L_{t}^{\prime \prime } =\frac{\tau }{t}L_{\tau }+\frac{t-\tau -\xi }{t}%
L_{\tau +\xi ,t+\xi } \\
&\Lambda _{t}^{\prime \prime } =\frac{\tau }{t}\Lambda _{\tau }+\frac{t-\tau -\xi }{t}\Lambda _{\tau +\xi ,t+\xi }.
\end{aligned}
$$

This leads to one crucial upshot. Consider the process $\left\{ Y_{s}\right\} _{0\leq
s\leq t}$, defined on $\left\{ \tau <t\right\} $ by
$$
Y_{s}=
\begin{cases}
W_{s} &\text{for }s<t \\ 
W_{s+\xi } & \text{for }s\in \left[ \tau ,t\right],
\end{cases}
$$
which jumps at time $\tau $ from $W_{\tau }$ to the boundary of $B_{1}\left(
W_{\tau }\right)$. On $\left\{ \tau <t\right\}$ consider also $%
Z_{s}=W_{\tau +s}-W_{\tau }$ for $s\leq \xi$. This is a process starting
at $0$ observed until the first time it hits the boundary of $B_{1}\left( 0\right)$.
We consider $\left\{ Z_{s}\right\} _{s\leq \xi }$ process modulo rotations,
i.e. we write%
$$
\Xi :=\left[ \left( W_{\tau +s}-W_{\tau }\right) _{0\leq s\leq \xi }\right] 
$$
where $\left[ \cdot \right] $ denotes the equivalence class under the action of
rotational group on the whole path in $\R^3$. Since, the distribution of a Brownian motion 
on the boundary of a ball (when started at the centre of the ball) is the uniform harmonic 
measure on the sphere, the process $\left\{ Y_{s}\right\} $ and $%
\left( \xi ,\Xi \right) $ are independent under $\mathbb{P}$. 

The splitting of the Hamiltonian according to \eqref{lemma3.3pf0} is:
$$
tH\left( L_{t}\right) =t\left\langle L_{t}^{\prime },\Lambda _{t}^{\prime
}\right\rangle +\xi \left\langle L_{\tau ,\tau +\xi },\Lambda _{t}^{\prime
}\right\rangle +\xi \left\langle \Lambda _{\tau ,\tau +\xi },L_{t}^{\prime
}\right\rangle .
$$
We fix some constant $a>0$. Then on the events $\big\{\left\Vert \Lambda _{t}^{\prime
}\right\Vert _{\infty }\leq a\big\}$, $\big\{\left\Vert \Lambda _{\tau ,\tau +\xi
}\right\Vert _{\infty }\leq a\big\}$ and $\big\{\xi \leq \theta \leq 1\big\}$, we have%
\begin{equation}\label{lemma3.3pf0.5}
tH\left( L_{t}\right) \leq t\left\langle L_{t}^{\prime },\Lambda
_{t}^{\prime }\right\rangle +2a\leq t\left\langle L_{t}^{\prime \prime
},\Lambda _{t}^{\prime \prime }\right\rangle +2a.
\end{equation}
Note that $\left\Vert \Lambda _{t}^{\prime }\right\Vert _{\infty }>a$ implies $%
\left\Vert \Lambda _{t}\right\Vert _{\infty }>a$. Therefore,%
\begin{equation}\label{lemma3.3pf1}
\begin{aligned}
&\widehat\P_t\bigg\{ \xi \leq \theta ,\ \tau +\xi \leq t\bigg\}  \\
&\leq\widehat\P_t\bigg\{ \xi \leq \theta ,\ \tau +\xi \leq t,\ \left\Vert
\Lambda _{t}^{\prime }\right\Vert _{\infty }\leq a,\ \left\Vert \Lambda
_{\tau ,\tau +\xi }\right\Vert _{\infty }\leq a\bigg\}  \\
&\qquad+\widehat\P_t\bigg\{ \left\Vert \Lambda _{t}\right\Vert _{\infty
}>a\bigg\} +\widehat\P_t\bigg\{ \left\Vert \Lambda _{\tau ,\tau +\xi
}\right\Vert _{\infty }>a,\ \xi \leq 1\bigg\}.
\end{aligned}
\end{equation}%
We can estimate the first probability on the right hand side above, since by \eqref{lemma3.3pf0.5},
\begin{equation}\label{lemma3.3pf2}
\begin{aligned}
&\mathbb{E}\bigg\{ \exp \left[ tH\left( L_{t}\right) \right]  \, \1\big\{\xi \leq
\theta ,\ \tau +\xi \leq t,\ \left\Vert \Lambda _{t}^{\prime }\right\Vert_{\infty }\leq a,\ \left\Vert \Lambda _{\tau ,\tau +\xi }\right\Vert
_{\infty }\leq a\big\}\bigg\}  \\
&\leq \mathrm{e}^{2a}\,\, \mathbb{E}\bigg\{ \exp \left[ t\left\langle
L_{t}^{\prime \prime },\Lambda _{t}^{\prime \prime }\right\rangle \right]\, \1\{\xi \leq \theta\} \bigg\} .
\end{aligned}
\end{equation}
Furthermore, since the Hamiltonian $t\left\langle L_{t}^{\prime \prime },\Lambda _{t}^{\prime
\prime }\right\rangle $ is independent of $\xi $ under $\mathbb{P}$, we have,
\begin{equation}\label{lemma3.3pf3}
\begin{aligned}
&\E\bigg\{\exp\big[t\big\langle L_t^{\prime\prime},\Lambda _{t}^{\prime \prime }\big\rangle\big]\,\,\1\{\xi\leq \theta\}\bigg\}\\
&= \P\big(\xi\leq\theta\big)\,\, \E\bigg\{\exp\big[t\langle L_t^{\prime\prime},\Lambda _{t}^{\prime \prime}\big\rangle\big]\bigg\}\\
&= \P\big(\xi\leq\theta\big)\,\, \frac{\E\bigg\{\exp\big[t\langle L_t^{\prime\prime},\Lambda _{t}^{\prime \prime}\big\rangle\big]\,\1\{\xi\leq 1\}\bigg]}{\P\big(\xi\leq 1\big)} \\
&\leq\frac{\P\big(\xi\leq\theta\big)}{\P\big(\xi\leq 1\big)}\,\,\E\bigg\{\exp\big[t\big\langle L_{t+1},\Lambda _{t+1}\big\rangle\big]\bigg\}.
\end{aligned}
\end{equation}
Also, since,
$$
t\left\langle L_{t+1},\Lambda _{t+1}\right\rangle \leq 2\left\Vert \Lambda
_{1}\right\Vert _{\infty }+\left\Vert \Lambda _{1,t+1}\right\Vert _{\infty
}+t\left\langle L_{1,t+1},\Lambda _{1.t+1}\right\rangle ,
$$
we have the estimate,
$$
\begin{aligned}
\mathbb{E}\bigg\{ \exp \left[ t\left\langle L_{t+1},\Lambda
_{t+1}\right\rangle \right] \bigg\}
 &\leq \mathbb{E}\bigg\{ \e^{2\left\Vert \Lambda _{1}\right\Vert _{\infty }}\bigg\} \mathbb{E}\bigg\{
\mathrm{e}^{\left\Vert \Lambda _{t}\right\Vert _{\infty }}\mathrm{e}%
^{tH\left( L_{t}\right) }\bigg\} .
\end{aligned}
$$
Summarizing \eqref{lemma3.3pf1}-\eqref{lemma3.3pf3}, we have,
\begin{eqnarray*}
\widehat\P_t\bigg\{ \xi \leq \theta ,\ \tau +\xi \leq t\bigg\}  &\leq
&\frac{\mathbb{P}\left( \xi \leq \theta \right) }{\mathbb{P}\left( \xi \leq
1\right) }\mathbb{E}\bigg\{ \mathrm{e}^{2\left\Vert \Lambda _{1}\right\Vert
_{\infty }}\bigg\} \,\, \mathrm{e}^{2a}\,\, \widehat\E_t\bigg\{ \e^{\left\Vert \Lambda _{t}\right\Vert _{\infty }}\bigg\}  \\
&&+\widehat\P_t\bigg\{ \left\Vert \Lambda _{t}\right\Vert _{\infty
}>a\bigg\} +\widehat\P_t\bigg\{ \left\Vert \Lambda _{\tau ,\tau +\xi
}\right\Vert _{\infty }>a,\ \xi \leq 1\bigg\} .
\end{eqnarray*}%
By \eqref{lemma3.3pf-1}, we then have,
$$
\begin{aligned}
&\lim_{\theta \rightarrow 0}\sup_{t,x}\widehat\P_t\bigg\{ \xi \leq
\theta ,\ \tau +\xi \leq t\bigg\} \\
&\leq \sup_{t}\widehat\P_t\bigg\{
\left\Vert \Lambda _{t}\right\Vert _{\infty }>a\bigg\} +\sup_{t,x}\mathbb{%
\widehat{P}}_{t}\bigg\{ \left\Vert \Lambda _{\tau ,\tau +\xi }\right\Vert_{\infty }>a,\ \xi \leq 1\bigg\} .
\end{aligned}
$$
Since by Corollary \ref{expdecaysupnorm}, $\sup_{t}\widehat\P_t\big\{ \left\Vert \Lambda _{t}\right\Vert
_{\infty }>a\big\} \rightarrow 0$ for $a\rightarrow \infty$, we only have to prove that%
\[
\lim_{a\rightarrow \infty }\sup_{t,x}\widehat\P_t\bigg\{ \left\Vert
\Lambda _{\tau ,\tau +\xi }\right\Vert _{\infty }>a,\ \xi \leq 1\bigg\} =0.
\]%

This can again be proved in the same way exploiting that $\left\Vert \Lambda _{\tau ,\tau
+\xi }\right\Vert _{\infty }$ is a function of the equivalence class $\Xi $ and invoking the above
independence argument. We drop the details to avoid repetition. 
\end{proof}


\begin{lemma}\label{tausum}
 For every $\eta>0$ there exists $r_0=r_0(\eta)\in\N$ such that,  
\begin{equation}
\sup_{t\geq 1} \sum_{r\geq r_0} \, \widehat\P_t\bigg\{\tau_{r}(0)\leq \sqrt r\bigg\}\leq \eta.
\end{equation}
\end{lemma}
\begin{proof}
For any $r>0$, we again use the splitting of the Hamiltonian as before and use the estimates obtained in Lemma \ref{onboundary} to get,
$$
\big|H(L_t)-H(L_{\sqrt r, t+ \sqrt r})\big|\leq \frac{C} t  {\sqrt r} \|\Lambda_{\sqrt r}\|_\infty.
$$
Then,
$$
\begin{aligned}
\E\bigg\{\e^{t H(L_t)}\, \1\{\tau_r\leq \sqrt r\}\bigg\}&\leq \E\bigg\{\e^{C\sqrt r\|\Lambda_{\sqrt r}\|_\infty}\1\{\tau_r(0)\leq \sqrt r\}\bigg\} \,\,\E\bigg\{\e^{tH(L_{\sqrt r, t+ \sqrt r})}\bigg\}
\\
&=\E\bigg\{\e^{2C\sqrt r\|\Lambda_{\sqrt r}\|_\infty}\bigg\}^{1/2} \, \P\bigg\{\tau_r(0)\leq \sqrt r\bigg\}^{1/2}\, \E\bigg\{\e^{tH(L_t)}\bigg\}.
\end{aligned}
$$
The first expectation grows like $\exp\{C\sqrt r\}$ by \eqref{prooflemma3est2.55} and the probability $\P\big\{\tau_r(0)\leq \sqrt r\big\}$ decays like $\exp\{-Cr^{3/2}\}$. This proves the lemma.
\end{proof}

Finally we are ready to prove Theorem \ref{thm3}.

{\bf{Proof of Theorem \ref{thm3}:}}
Given any $\eta>0$, by Lemma \ref{onboundary}, we choose $\theta(\eta)>0$ so that
\begin{equation}\label{est1}
\widehat\P_t\bigg\{\xi_{r}(x)\leq \theta(\eta)\bigg\}\leq \eta/3.
\end{equation}
for any $t, r, x$. From lemma \ref{tausum} we choose $r_0(\eta/3)$ so that
\begin{equation}\label{est2}
\widehat \P_t\bigg\{\tau_r(0)\leq \sqrt r\,\mbox{for some}\, r\geq r_0(\eta)\bigg\}\leq \eta/3.
\end{equation}
For any given $\eps>0$, we pick $u(\eps,\eta)>0$ so that $u(\eps,\eta)\geq \max\big\{M(\eps), r_0(\eta/3)^2\big\}$ such that
\begin{equation}\label{est3}
 \frac C{\theta(\eta)}\exp{\big\{-u(\eps,\eta)/C\big\}}\leq \eta/3,
\end{equation}
where $C$ is the constant coming from Theorem \ref{thm4} and this does not depend on $\eps$ or $\eta$. Now, let us choose the 
radius $r_\eps$ as required in \eqref{radiusrequirement}. Then, for
\begin{equation}\label{est4}
|x|\geq k(\eps,\eta):= \max\bigg\{ M(\eps)+ r_\eps, u(\eps, \eta)^2\bigg\},
\end{equation}
 on the complement of the event in \eqref{est2}, we have $\tau_{r_\eps}(x)\geq \tau_{u(\eps,\eta)}^2(0)$. Hence, according to Theorem \ref{thm4}, we have 
\begin{equation}\label{est5}
\widehat \P_t\bigg\{L_t\in U_\eps(\mu_x),\, \xi_{r_\eps}(x)> \theta(\eta),\, \tau_{r_\eps}(x)> u(\eps, \eta)\bigg\}\leq \eta/3.
\end{equation}
Let us combine \eqref{est1}, \eqref{est2} \eqref{est5} and stare at the requirement \eqref{est4}. We have proved Theorem \ref{thm3}.
\qed

\section{Identification of the limiting distribution: Proof of theorem \ref{thm1}}\label{sectionthm1proof}

In this section we will finish the proof of  Theorem \ref{thm1}. Recall that the goal is to show that the distribution of $L_t$ under $\widehat \P_t$ converges towards $\delta_{\mu_Y}$, where $Y$ is an $\R^3$-valued random variable with probability density $\psi_0/\int \psi_0$, where we recall that $\mu_y$ is the probability measure on $\R^3$ with density $\psi_y^2(\cdot)=\psi_0^2(y+\cdot)$, and $\psi_0$ is the maximizer in \eqref{Pekarfor}. Recall that 
$\d(\cdot,\cdot)$ denotes the metric on the set of probability measures on $\R^3$ that induces the weak topology. Introduce, for any probability measure $\mu$ on $\R^3$, the set of locations of best coincidence with a shift of $\mu_0$:
\begin{equation}
A(\mu)=\Big\{y\in\R^3 \colon \d(\mu,\mu_y)=\inf_{x\in\R^3}\d(\mu,\mu_{x})\Big\}.
\end{equation}
Furthermore, let $Y(\mu)$ be the lexicographically smallest site in $A(\mu)$ (which belongs to $A(\mu)$, since $x\mapsto \mu_x$ is continuous and therefore $A(\mu)$ closed). We need a certain continuity property of the map $\mu\mapsto Y(\mu)$, see the following lemma. Recall that we write $U_\eta(\mathfrak m)$ for the $\eta$-neighborhood of  the orbit $\mathfrak m= \{\mu_0\star\delta_x\colon x\in \R^3\}$ in the metric $\d$. 

\begin{lemma}\label{lem-Barysmall} For any $\delta>0$, there is an $\eta>0$ such that $|Y(\mu)-Y(\widetilde \mu)|<\delta$ for any probability measures $\mu,\widetilde\mu\in U_{\eta/2}(\mathfrak m)$ satisfying $\d(\mu,\widetilde\mu)<\eta$.
\end{lemma}

\begin{proof} Given a $\delta>0$, we pick $\eta>0$ so small that, for any $y,\widetilde y\in\R^3$, we have $|y-\widetilde y|<\delta/3$ as soon as $\d(\mu_y,\mu_{\widetilde y})<2\eta$. 

Now we show that $\diam(A(\mu))<\delta/3$ for any $\mu\in U_\eta(\mathfrak m)$. Indeed, given such a $\mu$, and given two sites $y,\widetilde y\in A(\mu)$, we have $\d(\mu,\mu_y)<\eta$ and $\d(\mu,\mu_{\widetilde y})<\eta$. The triangle inequality implies that $\d (\mu_y,\mu_{\widetilde y})<2\eta$, and the choice of $\eta$ implies that $|y-\widetilde y|<\delta/3$. Hence, $\diam(A(\mu))<\delta/3$.

Now we prove the assertion of the lemma, still with the same $\eta$. Let $\mu,\widetilde\mu\in U_{\eta/2}(\mathfrak m)$ be given satisfying $\d(\mu,\widetilde\mu)<\eta$. We pick some $y\in A(\mu)$ and $\widetilde y\in A(\widetilde  \mu)$, then we have $\d(\mu,\mu_y)<\eta/2$ and $\d(\widetilde \mu,\mu_{\widetilde y})<\eta/2$. This implies, via the triangle inequality, that $\d(\mu_y,\mu_{\widetilde y})\leq \d(\mu_y,\mu)+\d(\mu,\widetilde\mu)+\d(\widetilde \mu,\mu_{\widetilde y})<\eta/2+\eta+\eta/2=2\eta$. According to our choice of $\eta$, this implies that $|y-\widetilde y|<\delta/3$. Since both $\diam(A(\mu))$ and $\diam(A(\widetilde \mu))$ are smaller than $\delta/3$, again the triangle inequality implies that $|Y(\mu)-Y(\widetilde \mu)|<\delta$.
\end{proof}

By Theorem \ref{thm-tubeweak}, we know already that $\lim_{t\to\infty}\widehat\P_t(L_t\notin U_\eta(\mathfrak m))=0$ for any $\eta>0$. Therefore, for proving Theorem \ref{thm1}, it suffices to prove that the distribution of $Y(L_t)$ under $\widehat \P_t$ weakly converges towards the measure with density  $\psi_0/\int\psi_0$. 

Now we fix a bounded measurable set $I\subset\R^3$ with Lebesgue-negligible boundary. Thanks to Theorem \ref{thm2}, Theorem \ref{thm1} will then follow from 
\begin{equation}\label{Yconvergence}
\lim_{t\to\infty}\widehat \P_t\big(Y(L_t)\in I\big)=\frac{\int_I \psi_0(x)\,\d x}{\int_{\R^3}\psi_0(x)\,\d x}.
\end{equation}

We will need one more technical fact: 
\begin{equation}\label{thm1proof2}
\lim_{\ell\to\infty}\limsup_{t\to\infty} \sup_{t_0\leq t} \widehat\P_t\big\{\big|W_{t_0}\big| \geq \ell\big\}=0.
\end{equation}
The proof of this follows the same line of argument as in the proof of Theorem \ref{thm2} and we omit the details to avoid repetition. Hence, by \eqref{prooflemma3est2.5}, 
for any sufficiently large $a>0$, any $\eta>0$, any $t_0>0$ and any $\ell>0$, if $t$ is large,
\begin{equation}\label{Sec4proof1.5}
\begin{aligned}
&\widehat \P_t\big(Y(L_t)\in I\big)\\
&= \frac 1 {Z_t} \E\Big[ \e^{t H(L_t)} \1\{Y(L_t)\in I\} \1\{L_t\in U_{\eta/4}(\mathfrak m)\}\\
&\qquad\qquad\qquad\qquad\1\{\|\Lambda_{t_0}\|_\infty\leq a\}\1\{|W_{t_0}|\leq \ell\}\Big]\\
&\qquad +   F_1(t,t_0,\ell)+ F_2(t,t_0,a) + F_3(t,\eta),
\end{aligned}
\end{equation}
where $F_1, F_2$ and $F_3$ are some functions that satisfy $\lim_{\ell\to\infty}\limsup_{t\to\infty} \sup_{t_0\leq t} F_1(t,t_0,\ell)=\limsup_{t_0\to\infty}\limsup_{t\to\infty} \sup_{t_0\leq t} F_2(t,t_0,a)=\lim_{t\to\infty}F_3(t,\eta)=0$,
for $a$ large enough. We abbreviate $F_{\ell, a, \eta}(t_0,t)=F_1(t,t_0,\ell)+ F_2(t,t_0,a) + F_3(t,\eta)$.

Now we introduce two threshold parameters $\eps, \delta>0$ and pick a small $\eta$, depending on $\eps$ and $\delta$, to be determined later, but at least as small as described in Lemma~\ref{lem-Barysmall}. 

Let us again invoke the convex decomposition 
$L_t=\frac{t_0}t L_{t_0}+ \frac{(t-t_0)} t L_{t_0,t}$. We require that $t$ is at least so large that $\d(L_t,L_{t_0,t})<\eta/8$, in particular
$$
L_{t_0,t}\in U_{\eta/8}(\mathfrak m) \quad\Longrightarrow \quad L_t\in U_{\eta/4}(\mathfrak m)\quad \Longrightarrow\quad L_{t_0,t}\in U_{\eta/2}(\mathfrak m)
$$
and therefore, according to Lemma~\ref{lem-Barysmall}, on the event $\{L_t\in U_{\eta/4}(\mathfrak m)\}$,
$$
Y(L_{t_0,t})\in I_{-\delta}\quad \Longrightarrow\quad Y(L_t)\in I\quad \Longrightarrow\quad Y(L_{t_0,t})\in I_\delta,
$$
where $I_{-\delta}=\{x\in I\colon \dist(x,I^{\rm c})>\delta\}$ is the $\delta$-interior of $I$ and $I_\delta=\bigcup_{x\in I}(x-\delta,x+\delta)^3$ is the $\delta$-neighbourhood of $I$. From now on, we proceed only
with the proof of  the upper bound  \lq$\leq$\rq in \eqref{Yconvergence}. We pick up from \eqref{Sec4proof1.5} and proceed as follows.
\begin{equation}\label{Sec4proof2}
\begin{aligned}
&\widehat \P_t\big(Y(L_t)\in I\big)\\
&\leq \frac 1 {Z_t} \E\Big[ \e^{ tH(L_t)}\1\{ |W_{t_0}|\leq \ell\} \1\{\|\Lambda_{t_0}\|_\infty\leq a\}\\
&\qquad\qquad \1\{Y(L_{t_0,t})\in I_{\delta} \} \1\{L_{t_0,t}\in U_{\eps}(\mathfrak m)\}\Big]
\, + F_{\ell, a, \eta}(t_0,t).
\end{aligned}
\end{equation}
Since we can and will additionally require that $\eta/2<\eps$, we already replaced $U_{\eta/2}$ by $U_{\eps}$.

As in earlier sections, the convex decomposition of $L_t$ leads to the splitting
\begin{equation}\label{HsplitSec4}
t H(L_t)= \frac {t_0^2}t H(L_{t_0})+ 2 \frac{t_0 (t-t_0)}{t} \big\langle L_{t_0}, \Lambda_{t_0,t}\big\rangle + \frac{(t-t_0)^2} t H(L_{t_0,t}),
\end{equation}
which we will substitute on the right-hand side of \eqref{Sec4proof2}. We will again estimate the first two summands on the right hand side of \eqref{HsplitSec4}. Note that the first term lies between $0$ and $at_0^2/t$ on the event $\{\|\Lambda_{t_0}\|_\infty\leq a\}$, since $H(L_{t_0})\leq\|\Lambda_{t_0}\|_\infty$. 

Now we want to argue that, with high probability as $t\to\infty$, we can replace $\Lambda_{t_0,t}$ in \eqref{HsplitSec4} by $\Lambda(\mu_y)$ by making an error no more than $O(\sqrt\eps)$ in the exponent on the right-hand side of \eqref{Sec4proof2}. Indeed, we will condition on $Y(L_{t_0,t})$ (substituting it by $y$) and integrate over $Y(L_{t_0,t})\in I_\delta$. Note that, on the event $\{Y(L_{t_0,t})=y \}$, we have that $\{L_{t_0,t}\in U_{\eps}(\mathfrak m)\}=\{L_{t_0,t}\in U_{\eps}(\mu_y)\}$ by definition of $Y(\cdot)$. Furthermore, we recall \eqref{prooflemma3est5}, which implies that for some $C>0$,
$$
\limsup_{t\to\infty}\frac 1t \log \sup_{y\in I_\delta}\widehat\P_t\bigg\{L_{t}\in U_{\eps}(\mu_y), \,\, \, \|\Lambda(L_{t})-\Lambda(\mu_y)\|_\infty> C\sqrt\eps\bigg\}<0.
$$
On the event $\{\|\Lambda(L_{t})-\Lambda(\mu_y)\|_\infty\leq C\sqrt\eps\}$, we use \eqref{prooflemma3est2} and \eqref{prooflemma3est2.5}, to see that, with high probability, $\Lambda_{t_0,t}$ differs from $\Lambda(\mu_y)$ by no more than $O(\sqrt\eps )$.
This gives that, uniformly in $y\in I_\delta$,
$$
\begin{aligned}
\frac 1{Z_t}\E\Big[ &\e^{ 2 \frac{t_0 (t-t_0)}{t} \langle L_{t_0}, \Lambda_{t_0,t}\rangle}\1\{ |W_{t_0}|\leq \ell\} 
\e^{ \frac{(t-t_0)^2} t H(L_{t_0,t})} \1\{L_{t_0,t}\in U_{\eps}(\mu_y)\}\Big]\\
&\leq \e^{O(\sqrt\eps t_0)}\frac 1{Z_t}\E\Big[ \e^{ 2 \frac{t_0 (t-t_0)}{t} \langle L_{t_0}, \Lambda(\mu_y)\rangle}\1\{ |W_{t_0}|\leq \ell\} \e^{ \frac{(t-t_0)^2} t H(L_{t_0,t})} \1\{L_{t_0,t}\in U_{\eps}(\mu_y)\}\Big]\\
&\qquad\qquad+ o(1),
\end{aligned}
$$
as $t\to\infty$. Substituting $y$ by $Y(L_{t_0,t})$ and integrating over $Y(L_{t_0,t})\in I_\delta$, we obtain that
\begin{equation}\label{Sec4proof2.75}
\begin{aligned}
&\widehat \P_t\big(Y(L_t)\in I\big)\\
&\leq\frac 1 {Z_t}\e^{O(t_0^2/t)} \e^{O(\sqrt\eps t_0)} \int_{I_\delta} \,\E\Big[ \e^{ 2 \frac{t_0 (t-t_0)}{t} \langle L_{t_0}, \Lambda(\mu_y)\rangle}\1\{ |W_{t_0}|\leq \ell\} 
\\
&\qquad \e^{ \frac{(t-t_0)^2} t H(L_{t_0,t})} \1\{Y(L_{t_0,t})\in \d y \} \1\{L_{t_0,t}\in U_{\eps}(\mathfrak m)\}\Big] + F_{\ell, a, \eta,\eps}(t_0,t).
\end{aligned}
\end{equation}
where we absorbed the additional error term (which vanishes as $t\to\infty$) in the notation $ F_{\ell, a, \eta,\eps}(t_0,t)$.

Integrating over $W_{t_0}$ and using the Markov property, we obtain
\begin{equation}\label{Sec4proof4}
\begin{aligned}
&\widehat \P_t\big(Y(L_t)\in I\big)\\
&\leq \frac 1 {Z_t}\e^{O(t_0^2/t)}\e^{O(t_0\sqrt \eps)}\int_{I_\delta} \int_{B_\ell} \, \E\Big[\e^{ 2 t_0 \langle L_{t_0}, \Lambda(\mu_y)\rangle} \1\{W_{t_0}\in \d x\}\Big] \\
&\qquad \qquad\qquad\E_x\Big[ \e^{ \frac{(t-t_0)^2} t H(L_{t-t_0})}\1\{Y(L_{t-t_0})\in \d y \} \1\{L_{t-t_0}\in U_{\eps}(\mathfrak m)\}\Big] \\
&\qquad\qquad+  F_{\ell, a, \eta,\eps}(t_0,t).
\end{aligned}
\end{equation}%

Recall the measure tilting argument \eqref{measchange} with the function $\psi^2_y$ and the ergodic theorem for the measure $\P_0^{\ssup{\psi_y}}$ with invariant density $\psi_y^2$, which gives that
\begin{equation}\label{thm1proof7}
\begin{aligned}
\E\Big[ \e^{ 2 t_0  \langle L_{t_0}, \Lambda(\mu_y)\rangle}\1\{W_{t_0}\in \d x\}\Big]
&=\frac{\psi_y(0)}{\psi_y(x)} \e^{ \lambda t_0/2} \P_0^{\ssup{\psi_y}}(W_{t_0}\in \d x) \\
&= \psi_y(0)\psi_y(x) \e^{\lambda t_0/2} \d x \,\, (1+o(1)),
\end{aligned}
\end{equation}
as $t_0\to\infty$, uniformly in $x\in B_\ell$ and $y\in I_\delta$. We now choose $t_0$ so large that the above $1+o(1)$ is below $\e^\delta$, where we recall that $\delta$ was a given threshold parameter. This gives
\begin{equation}\label{Sec4proof5}
\begin{aligned}
&\widehat \P_t\big(Y(L_t)\in I\big)\\
&\leq \frac 1 {Z_t}\e^{O(t_0^2/t)}\e^{O(t_0\sqrt\eps)}\e^\delta \e^{\lambda t_0/2}
\int_{I_{\delta}}\psi_0(y)\int_{B_\ell} \d x\, \psi_0(y-x) \\
&\qquad \qquad\E_x\Big[ \e^{ \frac{(t-t_0)^2} t H(L_{t-t_0})}\1\{Y(L_{t-t_0})\in \d y \} 
\1\{L_{t-t_0}\in U_{\eps}(\mathfrak m)\}\Big] \\
&\qquad\qquad\qquad\qquad+  F_{\ell, a, \eta,\eps}(t_0,t)\\
&\leq\frac 1 {Z_t}\,\,\,\e^{O(t_0^2/t)}\e^{O(t_0\sqrt\eps)}\e^\delta\,\,\e^{\lambda t_0/2} \\
&\qquad\qquad \E_0\Big[\int_{\R^3} \d x \,\psi_0(Y(L_{t-t_0})+x)\1\{Y(L_{t-t_0})+x \in I_\delta\}\\
&\qquad \e^{ \frac{(t-t_0)^2} t H(L_{t-t_0})}\psi_0(Y(L_{t-t_0})) \1\{L_{t-t_0}\in U_{\eps}(\mathfrak m)\}\Big] +  F_{\ell, a, \eta,\eps}(t_0,t)\\
&\leq \e^{O(t_0^2/t)}\e^{O(t_0\sqrt\eps)}\e^\delta \,\,\, \int_{I_\delta}\d x\,\psi_0(x) \\
&\qquad \frac{\E_0\Big[ \e^{ \frac{(t-t_0)^2} t H(L_{t-t_0})}\psi_0(Y(L_{t-t_0})) \1\{L_{t-t_0}\in U_{\eps}(\mathfrak m)\}\Big]}{ \e^{-\lambda t_0/2}\,Z_t}+ F_{\ell, a, \eta,\eps}(t_0,t),
\end{aligned}
\end{equation}
where we used Fubini's theorem and the fact that the distribution of $Y(L_{t-t_0})$ under $\P_x$ is identical to the distribution of $Y(L_{t-t_0})+x$ under $\P_0$ (note the shift-invariance of $U_{\eps}(\mathfrak m)$ and $H(L_{t-t_0})$). In the third step, we shifted the $x$-integration by $Y(L_{t-t_0})$.

The proof of the corresponding lower bound with $I_{-\delta}$ (recall that $I_{-\delta}=\{x\in I\colon \dist(x,I^{\rm c})>\delta\}$ denotes
the $\delta$-interior of the open bounded set $I$) replacing $I_\delta$ on the right hand side of \eqref{Sec4proof5} is analogous modulo slight changes. In fact, to carry out the lower bound corresponding to 
 the second step in \eqref{Sec4proof5}, we note that the shift of the $x$-integration by $y\in I_{-\delta}$ leads to an $x$-integration over a superset of the $\diam(I_{-\delta})$-interior of $B_\ell$, which contains $B_{\ell^\prime}$ for some large $\ell^\prime$. But $\ell^\prime$ can be picked so large that the error when replacing $B_{\ell^\prime}$ afterwards by $\R^3$ is absorbed in the $F_{\ell, a, \eta,\eps}(t_0,t)$ term. 

Hence, for any arbitrary open bounded set $I^{\prime }$ we have that $1\geq \mathbb{\widehat{P}}_{t}(Y\left( L_{t}\right) \in I^{\prime })$ is bounded from below by the 
right hand side of \eqref{Sec4proof5} with $I_{\delta }$ replaced by $I_{-\delta }^{\prime }$. This gives a lower bound for $\e^{-\lambda t_0/2}\,Z_{t}$,
and combining this with the upper bound \eqref{Sec4proof5}, we obtain, by letting first $t\rightarrow \infty$, followed by $\varepsilon \rightarrow 0$, and then $t_{0}\rightarrow\infty$
and $\ell \rightarrow \infty $%
$$
\limsup_{t\rightarrow \infty }\mathbb{\widehat{P}}_{t}(Y\left( L_{t}\right) \in
I)\leq \frac{\int_{I_{\delta }}\psi _{0}(x)\d x}{\int_{I_{-\delta }^{\prime
}}\psi _{0}(x)\d x}.
$$
As $I^{\prime }$ and $\delta$ are arbitrary, we get
$$
\limsup_{t\rightarrow \infty }\mathbb{\widehat{P}}_{t}(Y\left( L_{t}\right) \in
I)\leq \frac{\int_{I}\psi _{0}(x)\d x}{\int_{\mathbb{R}^{3}}\psi _{0}(x)\d x}.
$$
This concludes the proof of the upper bound of the claim in \eqref{Yconvergence}.

In order to prove the corresponding lower bound for \eqref{Yconvergence},%
 we argue essentially in the same way, except that in order to prove the
upper bound for $Z_{t}$, we have to use the already proved tightness of $%
Y\left( L_{t}\right) $ which implies that for arbitrary $\delta $ we can
find a bounded set $I^{\prime }\subset \mathbb{R}^{3}$ with $\mathbb{\widehat{P}}%
_{t}(Y\left( L_{t}\right) \in I^{\prime })\geq 1-\delta $ for all $t$ large enough. This ends the proof of Theorem \ref{thm1}.


\qed

{\bf{Acknowledgement:}} The authors would like to thank S. R. S. Varadhan for reading an early draft of the manuscript and many valuable suggestions. 
The authors also thank an anonymous referee for pointing out a mistake 
in Section 4 that led to a more elaborate version of this part. The third author acknowledges 
the financial support and the hospitality of the Institut f\"ur Mathematik, Universit\"at Z\"urich, during his visits to the department. 



        








\frenchspacing
\bibliographystyle{plain}

\end{document}